\documentclass[review,11pt]{elsarticle}

\usepackage{amssymb}
\usepackage{amsmath}
\usepackage{setspace}
\usepackage{xcolor}
\usepackage{float}
\usepackage{array}
\usepackage{amsthm}
\usepackage{multirow}
\usepackage{graphicx}
\usepackage{mathtools}
\usepackage[normalem]{ulem}
\usepackage{cancel}
\newtheorem{theorem}{Theorem}[section]

\newtheorem{proposition}[theorem]{Proposition}

\usepackage[margin=1in]{geometry}
\newcommand{\ipekmodify}[1]{{\color{black} #1}}
\newcommand{\mathmodify}[1]{{\color{black} #1}}

\newcommand{\pmp}{PMP}
\newcommand{\ssc}{SSC}

\newcommand{\sscs}{SSC}
\newcommand{\pdc}{PD-Clustering}
\newcommand{\pdcp}{PD-Clustering Problem}

\newcommand{\fc}{Fuzzy Clustering}
\newcommand{\fcp}{Fuzzy Clustering Problem}
\newcommand{\fcs}{FC}
\newcommand{\fcm}{FCM}

\newcommand{\Gr}{$ \mathbf{G} $}
\newcommand{\Es}{$ \mathbf{E} $}
\newcommand{\Vs}{$ \mathbf{V} $}
\newcommand{\Xs}{$ \mathbf{X} $}
\newcommand{\Grev}{$ \mathbf{G}=(\mathbf{E},\mathbf{V}) $}
\newcommand{\dik}[2]{$d(v_#1,x_#2)$}
\newcommand{\pik}[2]{$p_{#1#2}$}
\newcommand{\vi}[1]{$v_{#1}$}
\newcommand{\xk}[1]{$x_{#1}$}
\newcommand{\vleft}[1]{{{v_i,\,v_j,\,v_k,\,i\neq j\neq k \in V_{left}}}}
\newcommand{\vright}[1]{{{v_i,\,v_j,\,v_k,\,i\neq j\neq k \in V_{right}}}}

\journal{Annals of Operations Research}

\begin{document}

\begin{frontmatter}



\title{A Foundational Perspective for Partitional Clustering on Networks}


\author[inst1]{Derya Ipek Eroglu}

\affiliation[inst1]{organization={Computing Sciences, SUNY Brockport},
            addressline={350 New Campus Drive}, 
            city={Brockport},
            postcode={14420}, 
            state={New York},
            country={USA}}

\author[inst2]{Cem Iyigun}

\affiliation[inst2]{organization={Industrial Engineering, Middle East Technical University},
            addressline={Üniversiteler Mahallesi, Dumlupınar Bulvarı No:1}, 
            postcode={06800}, 
            city={Ankara},
            country={Turkey}}

\begin{abstract}
This study presents a theoretical analysis of partitional clustering on networks, analyzing both hard and soft assignment schemes with different objective functions. Cluster centers are not restricted to vertices but can also be located along the edges. We examine four key models: P-Median (PMP) and Sum of Squares Clustering (SSC) under hard assignment, and Probabilistic Distance Clustering (PDC) and Fuzzy C-Means (FCM) under soft assignment. Through mathematical analysis, we uncover structural properties that differentiate these models, such as the significance of assignment bottleneck points and the role of vertex-restricted solutions in determining optimal cluster centers. Our findings reveal that, while SSC and FCM can yield optimal centers along edges, PMP and PDC inherently favor vertex placement, leading to insights into clustering behavior on networks. These insights offer new directions for designing efficient algorithms and have implications \ipekmodify{ranging from facility location and network design to clustering on the embedding graphs that power similarity search in modern retrieval systems}.

\end{abstract}


\end{frontmatter}


\section{Introduction}
\label{sec:1}

Clustering is an unsupervised learning method that groups similar data points and separates dissimilar ones based on a defined distance metric \cite{hart1975}. It is widely used as an exploratory analysis technique \cite{jain1999data}, uncovering underlying patterns and structures within complex datasets \cite{ClusIntro}. Clustering has been extensively studied across various domains, with existing literature primarily categorized into two main types: partitional clustering and hierarchical clustering. Partitional clustering partitions data into distinct, non-overlapping subsets, while hierarchical clustering organizes data into nested structures based on hierarchical relationships.

While traditional clustering methods are well-developed for continuous spaces, many real-world datasets are inherently relational, best represented as networks. This necessitates clustering methods that account for graph connectivity and discrete relationships. Consequently, graph clustering has gained increasing attention, with applications in social networks, biological systems, and transportation networks \cite{OuYang2017}. \ipekmodify{More recently, the same network abstraction has become central to similarity search over high-dimensional embeddings, where navigable graphs~\cite{HNSW} underpin the retrieval components of modern generative AI systems~\cite{RAG}.} Recent advances in clustering have expanded beyond standard partitioning techniques to incorporate more flexible models, such as soft clustering \cite{GraphClus} and density-controlled clustering \cite{Tariq2023}, making clustering more adaptable to real-world constraints.

Despite these advancements, many existing graph clustering approaches prioritize partitioning strategies rather than optimizing cluster center locations, limiting their applicability in certain theoretical and practical settings \cite{GraphClus}. However, location theory—a field concerned with optimal facility placement in networks—provides a natural framework to extend clustering models by incorporating optimization techniques used in network design and logistics. Recent work on p-median and facility location models \cite{Zaferanieh2024, Mahmoudi2023} highlights the importance of optimizing cluster centers, particularly in applications such as wireless sensor networks, urban facility placement, and transportation planning. Furthermore, clustering methods that integrate traffic-aware constraints \cite{Zaferanieh2024} and user-specified density parameters \cite{Tariq2023} have demonstrated the growing need for models that balance theoretical rigor with practical applicability.

In this paper, we utilize a unified theoretical framework for analyzing partitional clustering in networks, focusing on two key aspects: assignment schemes and objective functions. We examine both hard clustering, where each data point is assigned to a single cluster, and soft clustering, which allows fractional assignments and enables partial memberships. Our focus on soft clustering connects directly to recent advances which  unravels additional benefits coming with using soft assignment \cite{Hurtado2019}. Additionally, we explore clustering models that minimize either the sum of distances or the sum of squared distances between nodes and their assigned cluster centers. These formulations align with recent work on density-aware clustering \cite{Tariq2023} and p-median problem \cite{Zaferanieh2024}, both of which highlight the importance of the problems of interest. By establishing theoretical connections between clustering methodologies and location theory, we analyze the behavior of optimal solutions in networks, providing fundamental insights into the structural properties of these models. This theoretical foundation bridges recent advances in graph clustering and facility location models, offering a novel lens for understanding clustering in networks.

The remainder of this paper is structured as follows: Section 2 formally defines the clustering problems under study. Section 3 introduces the theoretical framework supporting our analysis, while Section 4 presents our key theoretical results. Finally, Sections 5 and 6 discuss our contributions and outline potential directions for future research.

\section{Problem Definitions and Network Setting}
\label{sec:2}

Clustering has been extensively studied in planar setting, providing crucial insights into data partitioning strategies. In this paper, we focus on four clustering models, distinguished by two primary factors: the assignment type (hard vs.\ soft) and the objective function (sum of distances vs.\ sum of squared distances). Table \ref{tab:twobytwo} summarizes these variations and their corresponding objective functions.

Let $I$ be the set of data points and $K$ the set of clusters. In the planar setting, each data point is represented by coordinates $x_i \in \mathbb{R}^d$, while cluster centers are denoted $c_k \in \mathbb{R}^d$. The function $d(x_i, c_k)$ typically refers to the Euclidean distance between the data points and cluster centers. In the hard assignment setting, $C_k$ denotes cluster $k$.

\begin{table}[h!]
    \caption{Categorization of Clustering Problems by Assignment Type and Objective Function with Mathematical Formulations}
    \label{tab:twobytwo}
    \centering
    \small 
    \renewcommand{\arraystretch}{1.5} 
    \setlength{\tabcolsep}{6pt} 
    \begin{tabular}{c c c c}
        \hspace{5pt} & \hspace{5pt}  & \multicolumn{2}{c}{\textbf{Assignment Type}} \\ 
        & & \textbf{Hard Assignment} & \textbf{Soft Assignment} \\ \cline{3-4}
        
        \multirow{2}{*}{\rotatebox[origin=c]{90}{\textbf{Objective Function}}}
        & \rotatebox[origin=c]{90}{\textbf{   Distance   }} 
        & \multicolumn{1}{|c|}{\begin{tabular}{c}
            \rule{0pt}{30pt}
            \textbf{\ipekmodify{P-median}} \\[3pt]
            $\ipekmodify{\min \sum_{k \in K}\sum_{i \in C_k}{d(x_i,c_k)}}$ \\
            \rule{0pt}{20pt}
        \end{tabular}} & 
        \multicolumn{1}{c|}{\begin{tabular}{c}
        \rule{0pt}{30pt}
            \textbf{PD-Clustering} \\[3pt]
            $ \min \sum_{i \in I}\sum_{k \in K}{p_{ik}^2d(x_i,c_k)} $ \\
            \rule{0pt}{20pt}
        \end{tabular}} \\ \cline{3-4}
        
        & \rotatebox[origin=c]{90}{\textbf{Squared Distance  }} 
        & \multicolumn{1}{|c|}{\begin{tabular}{c}
            \textbf{K-Means} \\[3pt]
            $\min \sum_{k \in K}\sum_{i \in C_k}{d(x_i,c_k)^2}$
        \end{tabular}} & 
        \multicolumn{1}{c|}{\begin{tabular}{c}
            \textbf{Fuzzy C-Means} \\[3pt]
            $ \min \sum_{i \in I}\sum_{k \in K}{p_{ik}^md(x_i,c_k)^2}$ 
        \end{tabular}} \\ \cline{3-4}
    \end{tabular}
\end{table}

\paragraph{\bfseries Hard assignment} Each data point belongs to exactly one cluster. Two well-known hard-assignment algorithms are:
\begin{itemize}
    \item \textbf{\ipekmodify{P-median}}, which minimizes the sum of distances between data points and their respective cluster centers. It is known for its robustness to outliers.
    \item \textbf{K-Means}, which minimizes the sum of squared distances to cluster centroids, emphasizing the reduction of within-cluster variance.
\end{itemize}

\paragraph{\bfseries Soft assignment} A data point can belong to multiple clusters with fractional memberships $p_{ik}$. Two prominent examples are:
\begin{itemize}
    \item \textbf{Fuzzy C-Means (\fcm)}~\cite{Fuzzy1}, which uses a fuzziness parameter $m$. As $m$ increases, cluster boundaries become more blurred, and membership values $p_{ik}$ approach a uniform distribution~\cite{FuzzySet}.
    \item \textbf{Probabilistic Distance Clustering (\pdc)}~\cite{PD1}, which also allows fractional memberships but minimizes the sum of distances, aligning it more closely with P-Median in terms of the objective function.
\end{itemize}

\ipekmodify{These four problems appear under different names across research communities, and we map the correspondence here to ease navigation. The distance-based hard assignment problem is known as K-Median in the machine learning and theoretical computer science literature, and as K-Medoids when cluster centers are restricted to data points; the widely used Partitioning Around Medoids (PAM) method of Kaufman and Rousseeuw~\cite{MedoBook} is an algorithm for the latter problem, although the two names are often used interchangeably in applied work. The squared-distance counterpart is the minimum sum-of-squares clustering problem, for which K-Means~\cite{Kmeans1} is the standard algorithm. In the location science literature, the network version of the distance-based problem is the P-Median Problem (\pmp). One distinction matters throughout this paper: K-Medoids and PAM restrict centers to data points, whereas the four problems we study allow centers anywhere on the network, at vertices or in the interior of edges. Whether this continuous relaxation changes the optimal solution is precisely the question we investigate.}

Although these formulations provide a strong theoretical foundation in planar contexts, many real-world datasets are naturally represented as networks. In such graph-based settings, the vertices represent the data points, and the edges define pairwise relationships. Hence, clustering on networks requires shortest path distances rather than Euclidean distances. As can be inferred from the comprehensive surveys~\cite{GraphClus,GraphClus2}, many network clustering approaches focus on partitioning the graph but do not explicitly address the problem of optimally locating cluster centers.

A key exception is the well-studied \pmp, an example of hard assignment on a network, which seeks to minimize the sum of shortest path distances between each \textbf{vertex(node)} and its assigned \textbf{center}. Hakimi's pioneering results~\cite{H1,H22} showed that optimal centers for the \pmp\ lie at vertices. Levy~\cite{L1} generalized this result and found that the root of this behavior is the objective function concavity. Building on Hakimi’s foundational insights ~\cite{H1,H22} and Levy’s extension~\cite{L1}, we introduce a unified framework that not only leverages these properties but also offers new perspectives, proofs, and a comparative treatment of four partitional clustering problems considering different assignment schemes using different objective functions. These foundational insights motivate to investigate different clustering problems to analyze the behavior of the objective functions and derive the structural properties of selecting clusters centers for each underlying problem. So the contribution of this work can be given in four folds : 
\begin{itemize}
\item [\textbf{{(i)}}] to present a unified theoretical framework for analyzing partitional clustering in networks, focusing on two key aspects: assignment schemes and objective functions,
\item [\textbf{(ii)}] to analyze the behavior of the objective function of the clustering problem under different settings,
\item [\textbf{(iii)}] to derive the structural properties of the cluster centers, and 
\item [\textbf{(iv)}] to lead developing efficient algorithms for solving the clustering problems with using the structural properties of the problem.  
\end{itemize}

To adapt the four clustering models given in Table~\ref{tab:twobytwo} to a network, we denote vertices as $v_i$ and restrict each cluster center $x_k$ to be on the network (where it can be at the nodes or on the edges). The distance function $d(v_i, x_k)$ then refers to the shortest path distance between vertex $v_i$ and the center $x_k$. In the sections that follow, we analyze how these four models behave in networks, highlighting both shared and distinctive structural properties.

\section{Fundamentals and Background}

In this section we introduce the basic notation and some concepts that will be used in the analysis of the models.

The clustering problems of interest are defined on an undirected, connected graph \Grev, where \Vs\ denotes the set of vertices and \Es\ represents the set of edges. Since \Gr\ is connected, it follows that $|\mathbf{E}| \geq |\mathbf{V}| - 1$, ensuring that every vertex is reachable. 

On this graph, we define \dik{i}{k} as the \textbf{shortest path distance} between vertex \vi{i} and cluster center \xk{k}. Since \Gr\ is undirected, it follows that \dik{i}{k} = $d(x_k, v_i)$. 
\ipekmodify{Throughout this study, the analysis is conducted with respect to this shortest path distance, induced by strictly positive edge lengths. The edge lengths themselves may be derived from any application-relevant measure, such as the Euclidean distance between endpoints, travel time, or cost. Our results rely only on the metric properties of the induced network distance, in particular the triangle inequality, and not on any property specific to Euclidean geometry. The squared-distance objectives in Table~\ref{tab:twobytwo} are defined as squares of this same network distance.} Table \ref{tab:notHardAss} summarizes the notation used throughout the paper.

\begin{table}[h]
	\caption{Table of Notations}
	\label{tab:notHardAss}
	\centering
	\begin{tabular}{l l} 
		\hline
		\Vs & Set of vertices \\
		\Xs & Set of cluster centers \\
            $ \mathbf{I} $ & Index set of vertices $ I=\{1,2,...,N\} $\\
		$N$ &  Number of vertices, $ |\mathbf{V}| $ \\
            $P$ & Number of clusters, $|\mathbf{X}|$ \\
		$ v_i $ & Vertex i, where $ i\in\mathbf{I}  $		\\
		$h_{i}$ & Weight of $v_i$, where ($h_{i}>0$, $ \forall i\in \mathbf{I}$ ) \\ 
		$ b_i $ & Arc bottleneck point of $ v_i $ on the edge $ (v_p,v_q) $\\
		$ a_i $ & Assignment bottleneck point of $ v_i $ \\	
		\xk{k} & Location of cluster center $k$ \\
		\dik{i}{k} & Length of the shortest path from \vi{i} to $x_k$ \\
		\pik{i}{k} & Membership probability of assignment of \vi{i} to cluster $k$ \\
		\hline
	\end{tabular}
\end{table}

By incorporating these notations, we extend the clustering models from Table \ref{tab:twobytwo} to networks by adding the center location constraint and updating the distance function. We next establish two key concepts—arc bottleneck points and assignment bottleneck points—to characterize the solution behavior on networks.

\subsection{Arc Bottleneck Point}

Let \vi{i} be an arbitrary vertex in \Gr, and let $e_{pq}$ be an edge connecting vertices $v_p$ and $v_q$ with length $l_e$. For any point $x \in e_{pq}$, the shortest path distance from \vi{i} to $x$ is given by:

\begin{equation}
d(v_i,x)=\min\{d(v_i,v_p)+d(v_p,x),d(v_i,v_q)+d(v_q,x)\}.
\end{equation}
\begin{figure}
	\centering
	\includegraphics[width=0.4\linewidth]{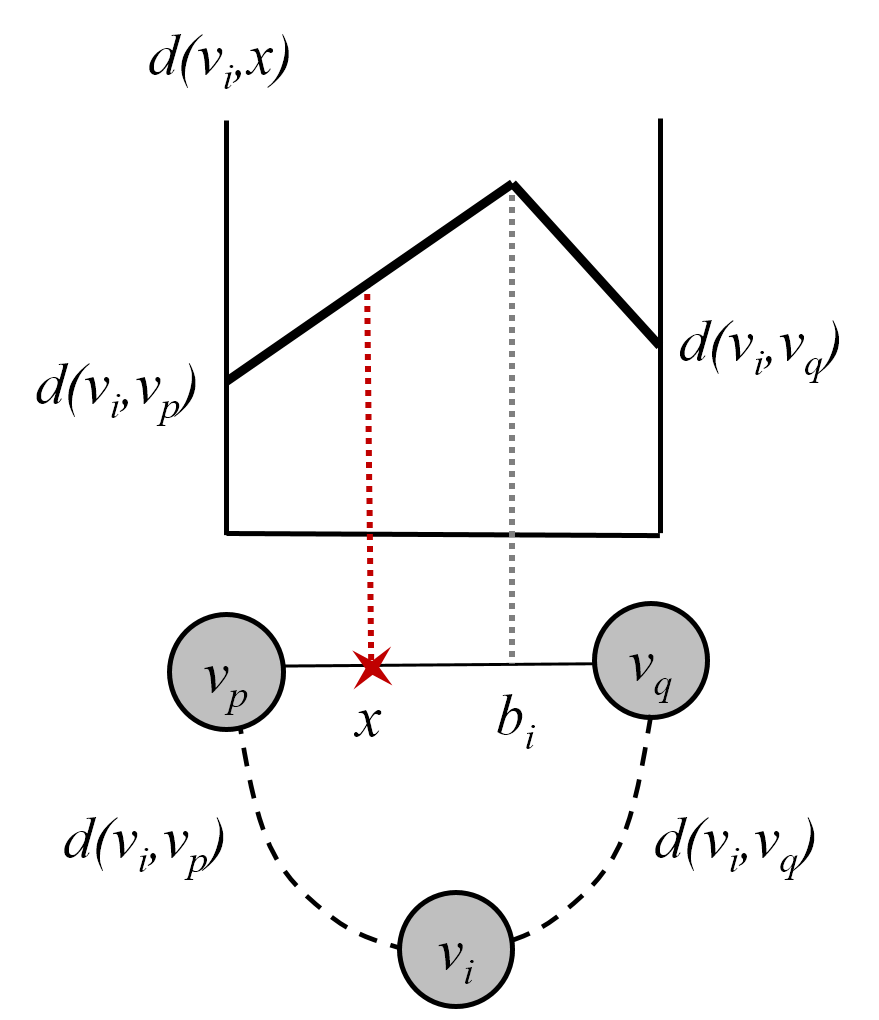}
	\caption{Distance function \dik{i}{} on the edge connecting \vi{p} and \vi{q}}
	\label{fig:arcbottleneck}
\end{figure}

This formulation implies that the shortest path to $x$ passes through either $v_p$ or $v_q$. Figure \ref{fig:arcbottleneck} illustrates this relationship. There exists a point, $b_i$, on edge $e_{pq}$ at which the shortest paths via $v_p$ and $v_q$ are of equal length. This point, called as \textbf{arc bottleneck point}, represents the farthest location on $e_{pq}$ from \vi{i} that maintains the shortest path property. The location of $b_i$ is determined by:

\begin{equation*}
d(v_i,v_p)+d(v_p,b_i)=d(v_i,v_q)+d(v_q,b_i) 
\end{equation*}

Substituting $d(v_q,b_i) = l_e - d(v_p,b_i)$, we obtain:

\begin{align}
d(v_i,v_p)+d(v_p,b_i)=&d(v_i,v_q)+l_e-d(v_p,b_i) \nonumber \\
d(v_p,b_i)=& \frac{1}{2} ( d(v_i,v_q)+l_e-d(v_i,v_p) ) \label{eqn:arcBot},
\end{align}

The arc bottleneck point $b_i$ has been studied in facility location problems. Before the study in \cite{FDS1}, Hakimi used this concept in his proof in \cite{H1}. There are three cases regarding the value of $d(v_p,b_i)$:

\begin{itemize}
	\item [\textbf{Case 1}]. If $ d(v_p,b_i) \leq0 $, the shortest path from \vi{i} to $x$ always passes from \vi{q}.
	\item [\textbf{Case 2}]. If $ d(v_p,b_i) \geq l_e $, the shortest path from \vi{i} to $x$ always passes from \vi{p}.
	\item [\textbf{Case 3}]. If $ d(v_p,b_i) \in (0,l_e) $, the shortest path from \vi{i} to $x$ passes from:
	\begin{itemize}
		\item \vi{p} if $d(v_p,x)<=d(v_p,b_i)$,
		\item \vi{q} if $d(v_p,x)>d(v_p,b_i)$.
	\end{itemize}
\end{itemize}

For cases with multiple vertices, additional bottleneck points can emerge, as illustrated in Figure \ref{fig:arcbottleneck2}. Since end vertices do not create bottleneck points, the number of arc bottleneck points on an edge is at most \ipekmodify{$N-2$}.

\begin{figure}
	\centering
	\includegraphics[width=0.45\linewidth]{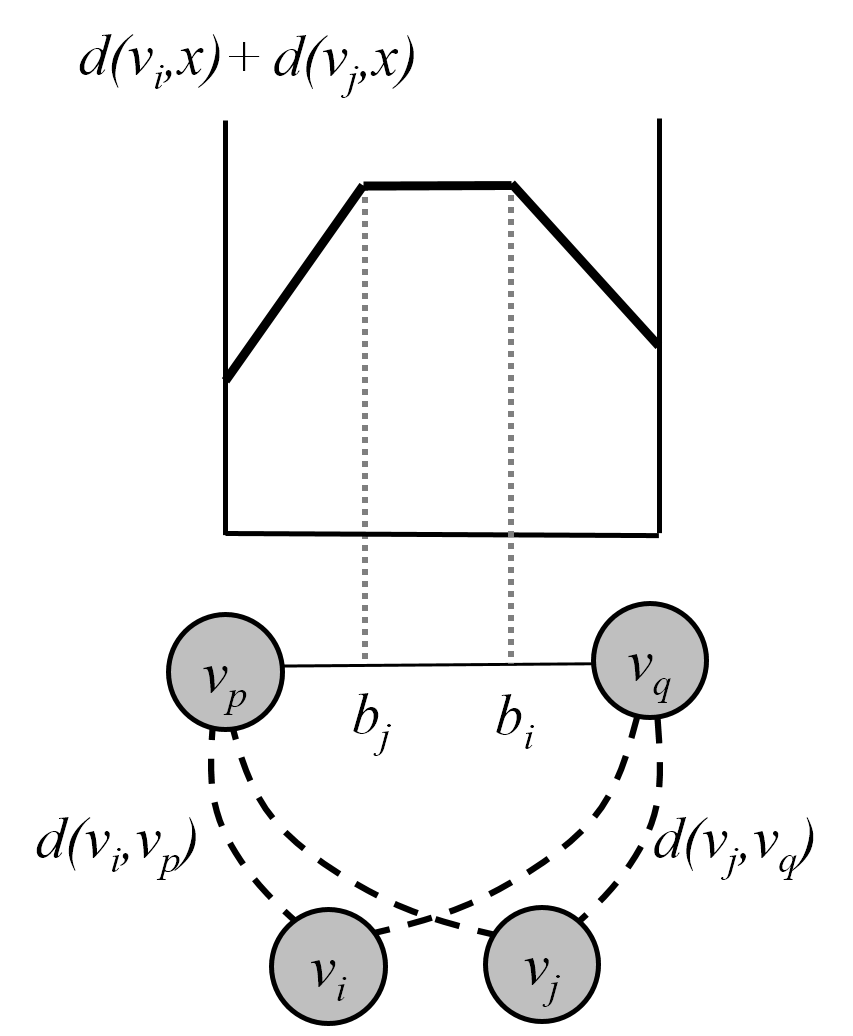}
	\caption{Distance function $d(v_i,x)+d(v_j,x)$ on the edge connecting $v_p$ and $v_q$}
	\label{fig:arcbottleneck2}
\end{figure}

\subsection{Assignment Bottleneck Point}

While arc bottleneck points capture the farthest shortest-path location on an edge, they do not account for cluster assignment changes when centers move. We introduce the assignment bottleneck point, which identifies the location where a vertex switches its nearest cluster center.

Let \vi{i} be an arbitrary vertex in \Gr, and let \xk{k} $\in \mathbf{E}, k=1,...,P$ be the closest center to \vi{i}. Assume that \xk{l} $\in$ \Gr\ is the second-closest center to \vi{i}, and that all cluster center locations except \xk{k} remain fixed. As \xk{k} moves along $e_{pq}$, its distance to \vi{i} changes, and at a critical point, \xk{l} may become closer than \xk{k}. The location where this assignment shift occurs is defined as the assignment bottleneck point. Figure \ref{fig:assbottleneck} illustrates an example where an assignment bottleneck point $a_i$ appears. Initially, \dik{i}{k} = 13 and \dik{i}{l} = 15, with \xk{k} as the closest center. As \xk{k} moves toward \vi{p}, its distance to \vi{i} increases. If \xk{k} moves more than 2 units, the assignment changes to \xk{l}, as \dik{i}{k} $\geq$ \dik{i}{l}.

\begin{figure}
	\centering
	\includegraphics[width=0.45\linewidth]{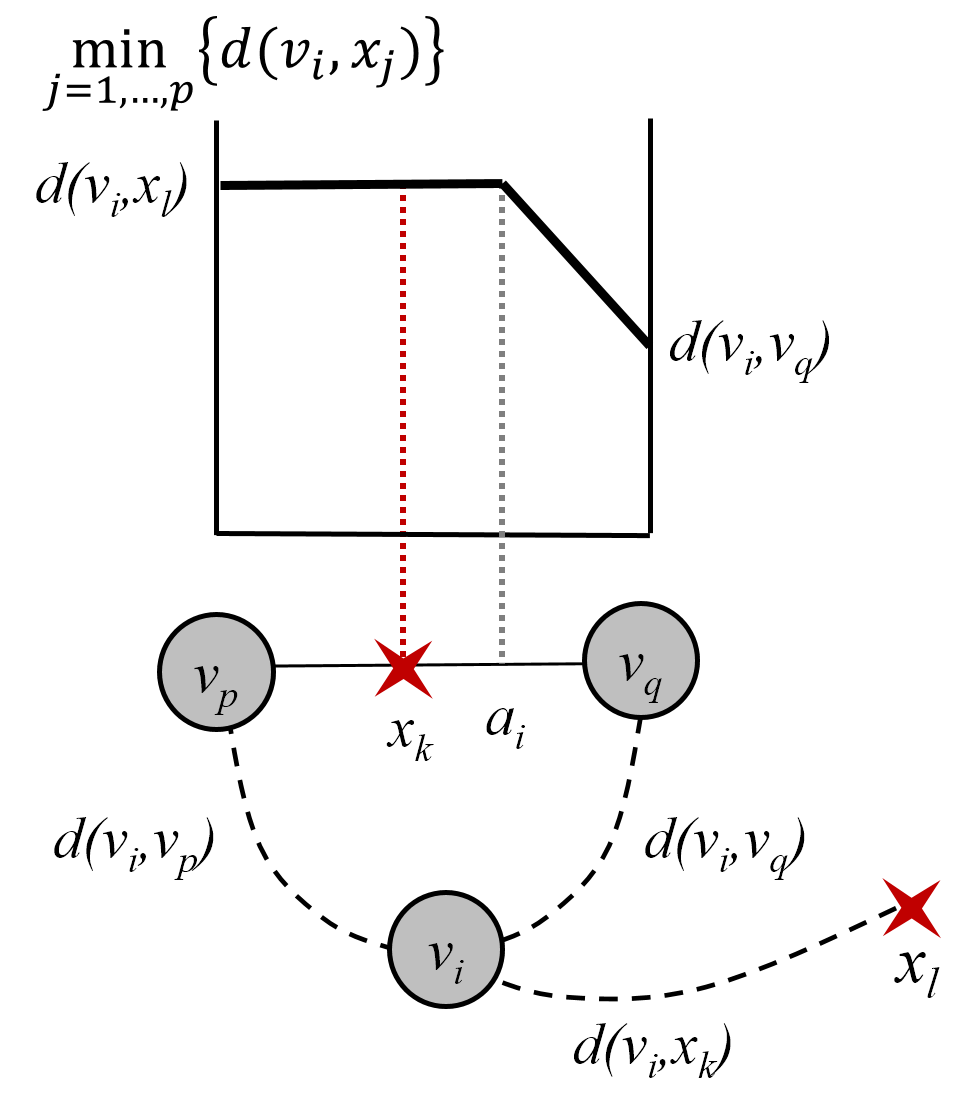}
	\caption{Assignment Bottleneck Point}
	\label{fig:assbottleneck}
\end{figure}

If an arc bottleneck point exists on $e_{pq}$, multiple assignment bottleneck points may arise. Figure \ref{fig:assbottleneck2} illustrates a scenario where:
1. The assignment switches at $a_i^q$, making \xk{l} the closest center.
2. The assignment switches back at $a_i^p$, reassigning \xk{k} as the closest center.

\ipekmodify{
\begin{proposition}\label{prop:abpbound}
Let $e_{pq}$ be an edge of length $l_e$ that realizes the shortest path distance between its end vertices, that is, $d(v_p,v_q)=l_e$. When a single cluster center moves along $e_{pq}$ while all other centers remain fixed, each vertex $v_i \in \mathbf{V} \setminus \{v_p,v_q\}$ creates at most two assignment bottleneck points on $e_{pq}$, and it creates two only if it has an arc bottleneck point in the interior of $e_{pq}$. Each end vertex creates at most one. Consequently, $e_{pq}$ contains at most $2(N-2)+2$ assignment bottleneck points.
\end{proposition}
\begin{proof}
Let $t \in [0,l_e]$ denote the position of the moving center $x_k$ on $e_{pq}$, so that $d(v_i,x_k)=\min\{d(v_i,v_p)+t,\, d(v_i,v_q)+l_e-t\}$. As a function of $t$, this distance is piecewise linear with slopes $+1$ and $-1$ and at most one breakpoint, the arc bottleneck point $b_i$; it is increasing before $b_i$ and decreasing after. Since only $x_k$ moves, the distance from $v_i$ to its second-closest center is constant, and the assignment bottleneck points of $v_i$ are exactly the positions at which $d(v_i,x_k)$ crosses this constant level. A piecewise linear function with at most one interior maximum crosses a constant level at most twice, and twice only if the interior maximum exists, that is, only if $b_i$ lies in the interior of $e_{pq}$. For the end vertex $v_p$, the assumption $d(v_p,v_q)=l_e$ gives $d(v_p,x_k)=\min\{t,\, d(v_p,v_q)+l_e-t\}=t$, a monotone function, which crosses a constant level at most once; the same argument applies to $v_q$. Summing over the $N-2$ remaining vertices and the two end vertices yields the bound.
\end{proof}
}

\begin{figure}
	\centering
	\includegraphics[width=0.45\linewidth]{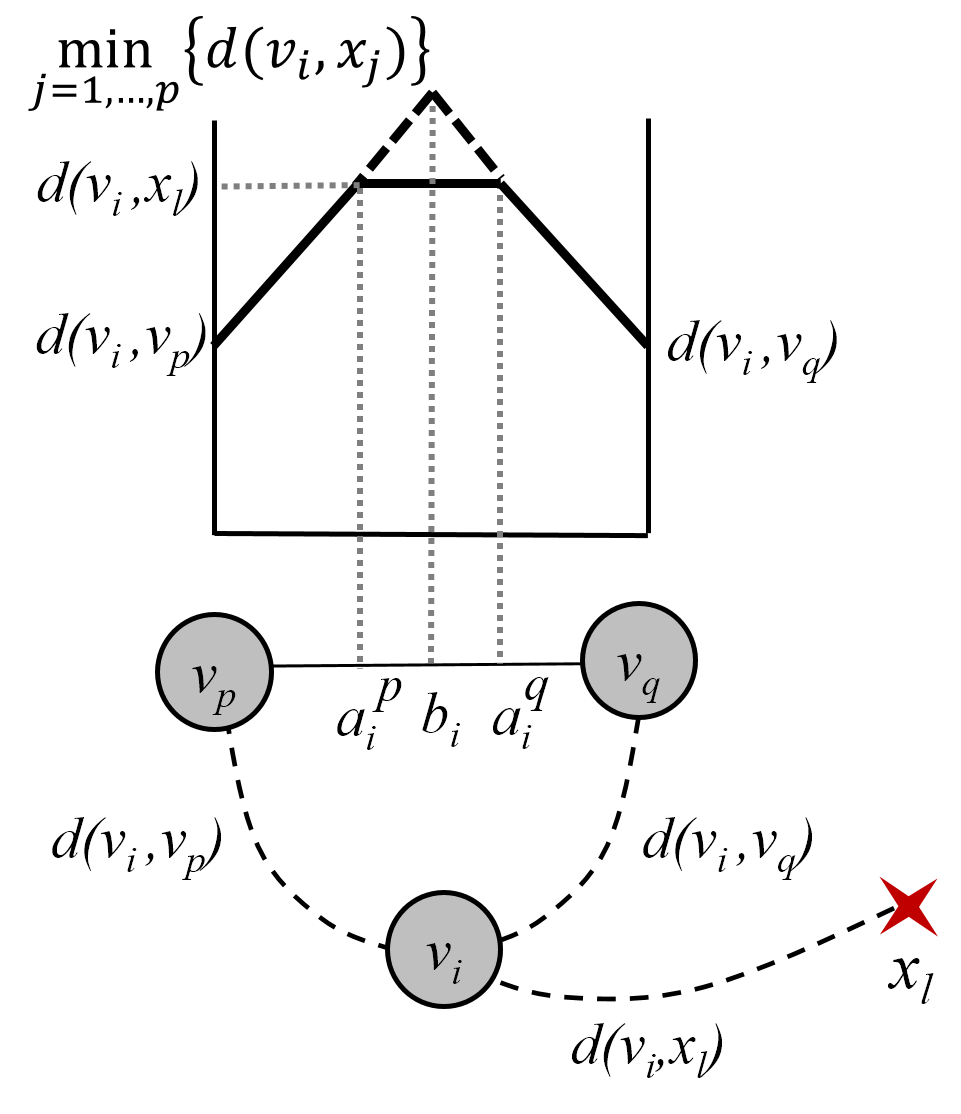}
	\caption{Assignment Bottleneck Point with Arc Bottleneck Point}
	\label{fig:assbottleneck2}
\end{figure}

\ipekmodify{
Having discussed the fundamental definitions for both hard and soft assignment clustering in network settings, next section focuses on the core theoretical aspects of these models. We analyze the objective function behavior of all four problems and establish the theorems that underpin our analytical results. This framework will clarify how and why optimal solutions tend to concentrate on specific locations of \Gr, ultimately revealing deeper insights into the nature of cluster center placement in networks.
}
\section{Structural Insights and Analytical Results}

In this section, we conduct our analysis under hard and soft assignment schemes and different objective functions as described in Table \ref{tab:twobytwo}.

\subsection{Hard Assignment Problems}

In the hard assignment case, each vertex is assigned to a single cluster, determined by its closest cluster center. Two key hard clustering problems are examined in this section: the \pmp, which minimizes the sum of distances between vertices and their assigned cluster centers, and the \sscs, which minimizes the sum of squared distances. We begin with the \pmp\ and extend the discussion to the \sscs, analyzing their theoretical properties in a network setting.

\subsubsection{P-Median Problem}
The \pmp\ is formally defined as
\begin{align*}
	\text{\textit{minimize} }&   
        f(\mathbf{X})=\sum_{k=1}^P\sum_{i \in C_k} h_i{d(v_i,x_k)} \\
	&\text{subject to } \nonumber \\
	&x_k \in \mathbf{G} \quad \forall \, k=1,...,P,
\end{align*} 

where \Xs\ represents the set of cluster centers, \xk{k} is the decision variable for the location of center $k$, and $h_i$ is a nonnegative constant weight associated with vertex \vi{i}. Since hard assignment is imposed, each vertex is exclusively assigned to one cluster, and only the distance to its closest center contributes to the objective function. We first analyze the case with a single cluster (the 1-Median Problem) and subsequently extend the results to the general \pmp.

\paragraph{\bfseries 1-Median Problem}\label{ss:pmedSS}
Suppose we have only one cluster and one cluster center will be located on \Gr. In that case, the formulation is
\begin{align}\label{eqn:pMed1obj}
\text{\textit{minimize}  }&  f(x_c)=\sum_{i=1}^N h_id(v_i,x_c)\\
&\text{subject to } \nonumber \\
&x_c \in \mathbf{G}.
\end{align}

Consider a simple line graph with four vertices as in Figure \ref{fig:graphline4}. If the cluster center is located anywhere on the graph, the objective function remains constant. This implies that the optimal center location could be on either a vertex or an edge.

\begin{equation*}
f=(a+y)+y+(l-y)+(b+l-y)=a+2l+b,
\end{equation*}

\begin{figure}[htbp]
	\centering
	\includegraphics[height=2cm]{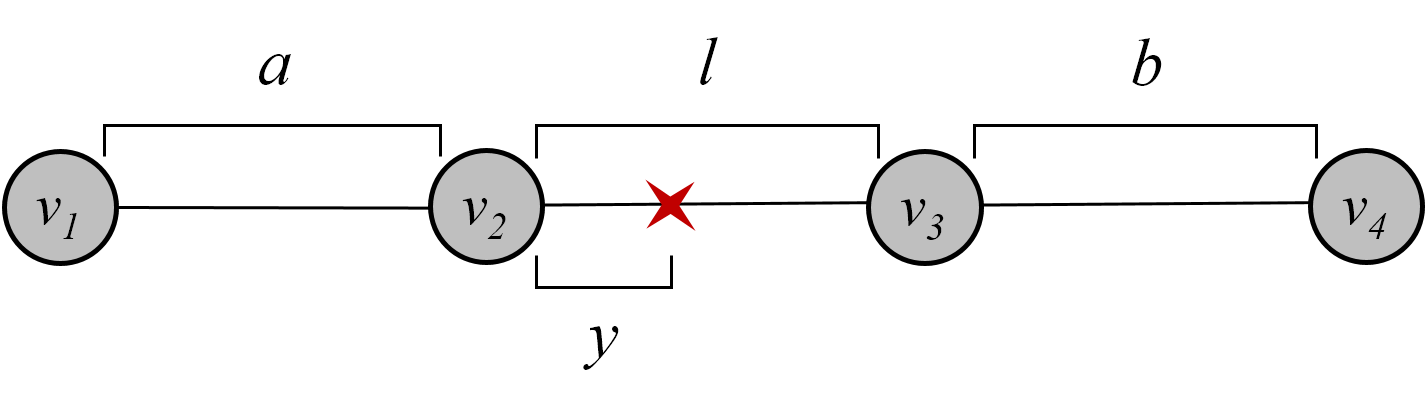}
	\caption{A Line Graph with 4 vertices}
	\label{fig:graphline4}
\end{figure}
\begin{figure}[htbp]
\centering
\includegraphics[width=0.4\linewidth]{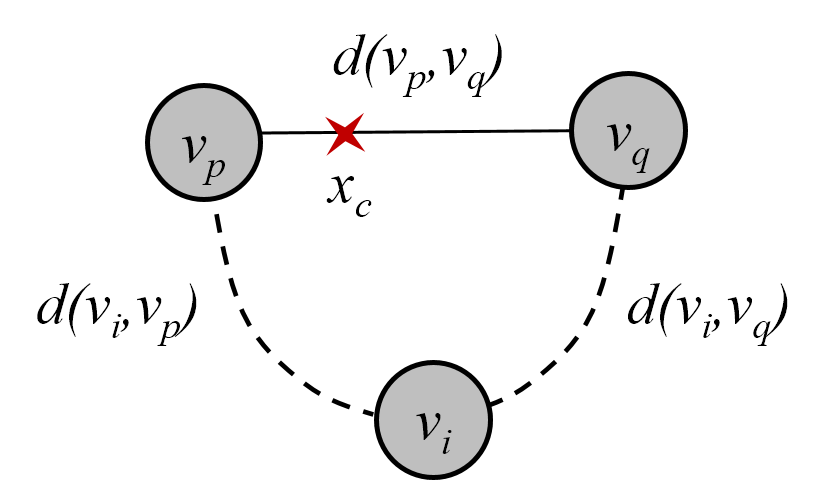}
\caption{An illustration of a part of a graph \Gr}
\label{fig:grgeneral}
\end{figure}
\begin{figure}[htbp]
	\centering
	\includegraphics[width=1\linewidth]{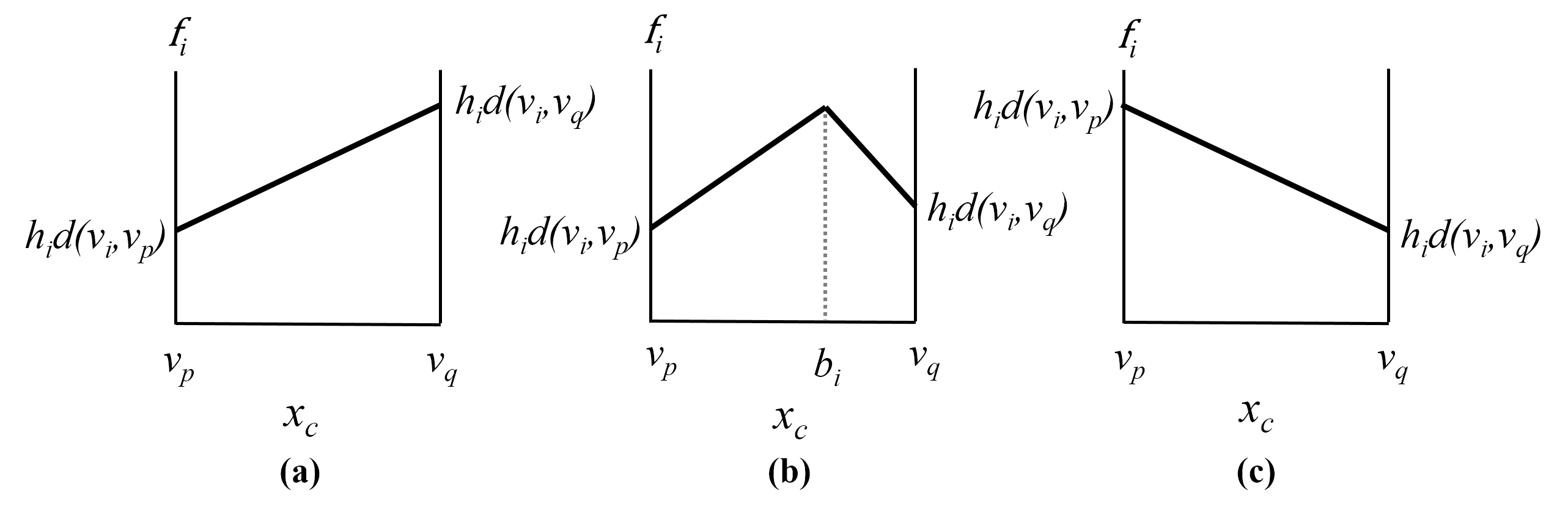}
	\caption{Objective function component for \vi{i} (denoted as $f_i$) when \xk{c} is moved along the edge $(v_p,v_q)$}
	\label{fig:pmed1obj}
\end{figure}

A general network is illustrated in Figure \ref{fig:grgeneral} in which cluster center \xk{c} is on an edge $e_{pq}$ connecting $(v_p,v_q)$. In the case of this network, we may observe three different patterns of objective function component of a vertex \vi{i} to the \eqref{eqn:pMed1obj} depending on location of \xk{c}. The shortest path from \vi{i} to \xk{c} may pass through \vi{p} or \vi{q} regardless of the location of \xk{c}, which also means that there is no arc bottleneck point. If there is an arc bottleneck point on $e_{pq}$, the shortest path to \xk{c} will pass from \vi{p} or \vi{q} depending on the location of \xk{c} on $e_{pq}$. These patterns are shown in Figure \ref{fig:pmed1obj}. In the Figure, (a) is the case when shortest path to \xk{c} passes from \vi{p}, and (c) is the case when shortest path to \xk{c} passes through \vi{q}. In both cases, there is no arc bottleneck point. In (b), if \xk{c} is in the interval $[v_p,b_i]$, shortest path to \xk{c} passes through \vi{p}; otherwise, shortest path to \xk{c} passes through \vi{q}. The reason of this behavior is the bottleneck point $b_i$ observed. In all of these cases, it could be observed that the objective function is linear or piecewise concave.

Using this structure, Hakimi \cite{H1} proved that the optimal center locations always lie in \Vs. Levy \cite{L1} extended this proof by showing that this result is a consequence of the concavity of the objective function. Since the summation of concave functions remains concave, the optimal solution to the 1-Median problem is always located at a vertex.

\paragraph{\bfseries P-Median Problem}\label{ss:pmedMS}

In \pmp, vertices are assigned to the clusters with the closest cluster center minimizing the p-median objective function in Table \ref{tab:twobytwo}.
 For the general \pmp, Hakimi \cite{H22} extended his previous proof and showed that the optimal cluster centers are always located in \Vs. His proof follows from decomposing the problem into $P$ separate 1-Median problems and showing that each center must be positioned at a vertex.

An alternative approach is to analyze the behavior of the objective function along the edges. Consider a network with \textit{P} cluster centers, where \xk{c} and \xk{k} are the closest and second-closest cluster centers to \vi{i}, respectively. If \xk{c} is moved along an edge $(v_p,v_q)$, the assignment of \vi{i} to a cluster may change. These assignment changes occur at assignment bottleneck points, as shown in Figure \ref{fig:pmedpobj}.

\begin{figure}[htbp]
	\centering
	\includegraphics[width=0.4\linewidth]{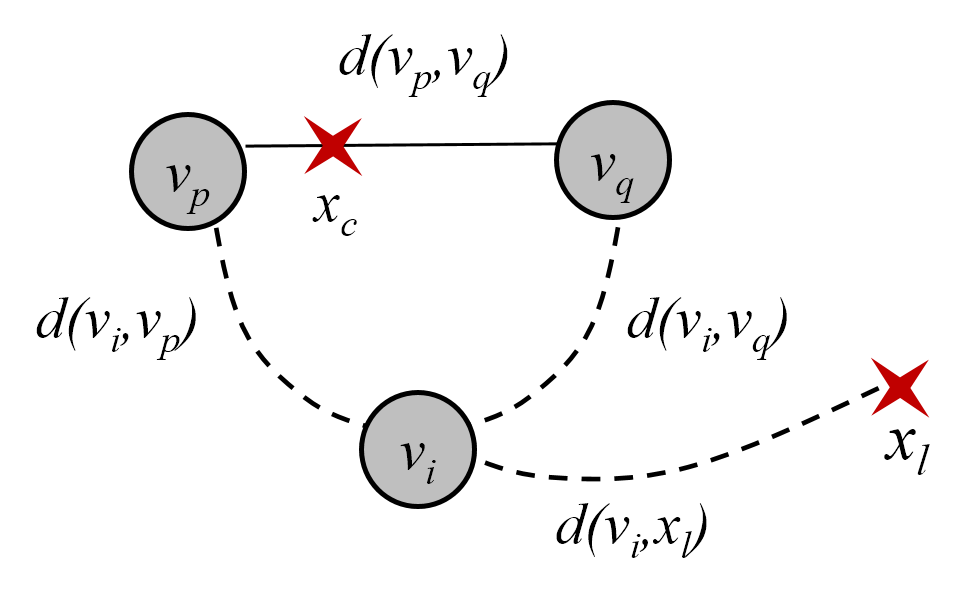}
	\caption{A visualization of a part of graph \Gr\, with 2 closest cluster centers}
	\label{fig:grgeneralass}
\end{figure}
Even with multiple centers, the objective function remains piecewise concave due to the presence of arc bottleneck points and assignment bottleneck points. Since the sum of concave functions is also concave, the optimal centers are always located in \Vs, as proved in \cite{L1}.
\begin{figure}[htbp]
	\centering
	\includegraphics[width=1\linewidth]{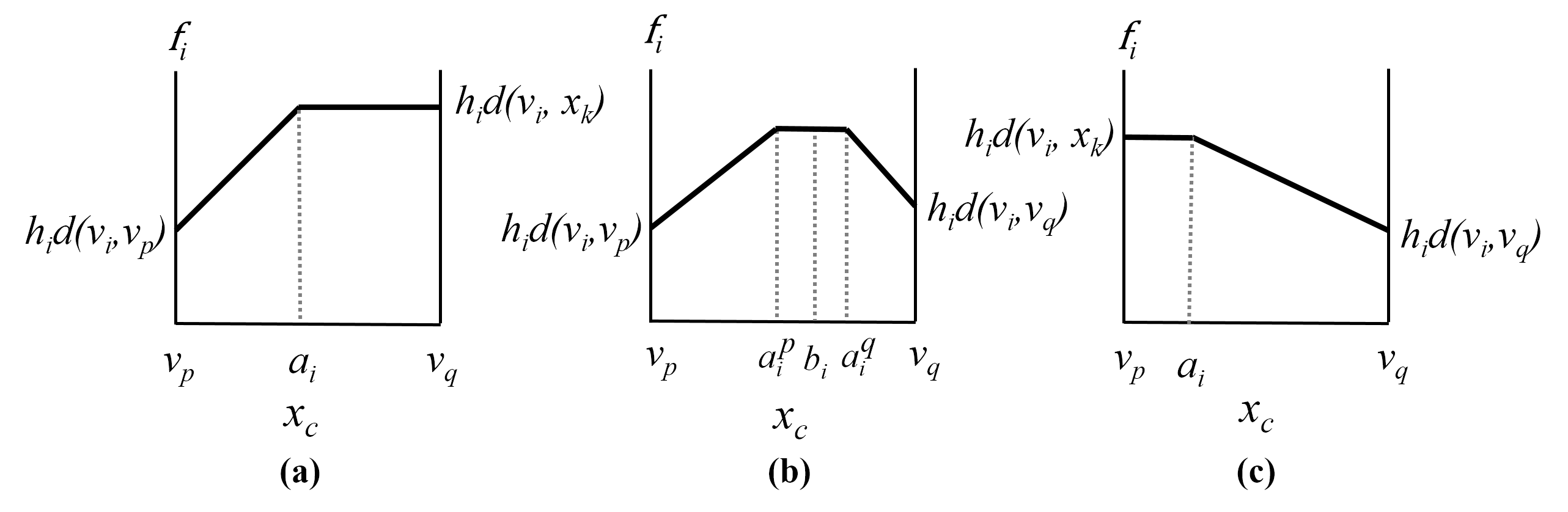}
	\caption{Objective function component for \vi{i} (denoted as $f_i$) when \xk{c} is moved along the edge $(v_p,v_q)$ and \xk{k} is the second closest cluster center to \vi{i}}
	\label{fig:pmedpobj}
\end{figure}

\subsubsection{Sum of Squares Clustering (\sscs) Problem on Networks}
\sscs\ problem differs from \pmp\ in that it uses sum of squared distances in the objective function instead of sum of distances. \sscs\ problem is defined as
\begin{align*}
\text{\textit{minimize}  }&  
 f(\mathbf{X})=\sum_{k=1}^P \sum_{i \in C_k} h_i{d(v_i,x_k)^2}\\
&\text{subject to } \nonumber \\
&x_k \in \mathbf{G} \quad \forall \, k=1,...,P,
\end{align*}

where \xk{k} is the decision variable representing the location of cluster center $ k $, and $h_i$ is a nonnegative constant weight associated with vertex \vi{i}. Due to hard assignment constraints, each vertex is assigned to exactly one cluster. 

A key difference from the \pmp\ is that, as observed by Carrizosa et al. \cite{VNS1}, the optimal center locations may lie not only in \Vs\ but also in \Es. Consequently, restricting cluster center locations solely to \Vs\ could lead to suboptimal solutions. In this subsection, we analyze this property in more detail.

\paragraph{\bfseries Sum of Squares Clustering on Networks with a Single Cluster}\label{ss:SSC1}

As in \pmp, when there is only one cluster, every vertex belongs to that cluster. Consider the simple line graph with four vertices shown in Figure \ref{fig:graphline4}. The objective function for the \sscs\ problem is given by:
\begin{equation*}
f=(a+y)^2+y^2+(l-y)^2+(b+l-y)^2.
\end{equation*}
This function is continuous and twice differentiable. Its first and second derivatives with respect to $y$ are:
\begin{align}
\frac{df}{dy}=&8y+2a-2l-2b-2l,\label{fDer} \\ 
\frac{d^2{f}}{d{y^2}}=&8, \label{secDer}
\end{align}

which shows that $f$ is a convex function of $y$. \eqref{fDer} gives us

\begin{equation*}
y=\frac{b+2l-a}{4}.
\end{equation*}

Therefore, the optimal center location depends on the following cases:

\begin{itemize}
    \item If $y \in (0,l)$, the optimal center is located within the edge $(v_2,v_3)$.
    \item If $y=0$, the optimal center is at vertex $v_2$.
    \item If $y=l$, the optimal center is at vertex $v_3$.
    \item If $y \in (0,-a)$, the optimal center is within the edge $(v_1,v_2)$.
    \item If $y=-a$, the optimal center is at vertex $v_1$.
    \item If $y \in (l,b+l)$, the optimal center is within the edge $(v_3,v_4)$.
    \item If $y=b+l$, the optimal center is at vertex $v_4$.
\end{itemize} 

Thus, in contrast to the \pmp, the optimal solution for \sscs\ is not necessarily at a vertex but may be within an edge.

Consider a generalized line graph, as in Figure \ref{fig:grgeneral}, where the cluster center \xk{c} is located on an edge $(v_p,v_q)$. The objective function is:

\begin{equation}
f=\sum_{i=1}^N{h_id(v_i,x_c)^2}.
\end{equation}
Using the shortest path properties,
\begin{equation*}
d(v_i,x_c)=\min \left \{ d(v_i,v_p)+d(v_p,x_c),d(v_i,v_q)+d(v_q,x_c) \right \}.
\end{equation*}
We assume vertices are ordered such that:
\begin{align*}
d(v_{i_j},x_c)&= d(v_{i_j},v_p)+d(v_p,x_c), for  \ j=1,...,r, \\
d(v_{i_j},x_c)&= d(v_{i_j},v_q)+d(v_q,x_c), for  \ j=r+1,...,N.
\end{align*}
 
Rewriting the objective function and taking derivatives:

\begin{align}
\frac{\partial{f}}{\partial{x_c}}=&\sum_{j=1}^N{h_{i_j}2d(v_p,x_c)}\nonumber \\ 
& +\sum_{j=1}^r
{2h_{i_j}d(v_{i_j},v_p)}-\sum_{j=r+1}^N
{2h_{i_j}(d(v_{i_j},v_q)+d(v_p,v_q))} \label{eqn:fDer2}\\
\frac{\partial^2{f}}{\partial{x_c^2}}=&\sum_{j=1}^N{2h_{i_j}}. \label{eqn:secDer2}
\end{align}

Since \eqref{eqn:secDer2} is positive, $f$ is convex. Setting \eqref{eqn:fDer2} to zero and solving for $d(v_p,x_c)$, we get:

\begin{equation*}
d(v_p,x_c)=\frac{\sum_{j=r+1}^N{h_{i_j}(d(v_{i_j},v_q)+d(v_p,v_q))}-\sum_{j=1}^r{h_{i_j}d(v_{i_j},v_p)}}{\sum_{j=1}^N{h_{i_j}}}.
\end{equation*}

From this equation, given that the center is to be located on $ (v_p,v_q) $, the following interpretations could be made. If $ d(v_p,x_c)\leq0 $, \xk{c} will be located on vertex \vi{p}. If $ d(v_p,x_c)\geq d(v_p,v_q) $, \xk{c} will be located on vertex \vi{q}. For other values of $ d(v_p,x_c) $, \xk{c} will be located on the edge $ (v_p,v_q) $. This confirms that the optimal center may lie on an edge rather than at a vertex.

Now we can generalize our results for a network \Gr\,with one cluster. Let us assume that the cluster center is on edge $(v_p,v_q)$ and $ f_i $ denotes the objective function component of vertex \vi{i}. There are three cases of $ f_i $ as shown in Figure \ref{fig:ssq1obj}. In (a) and (c), the shortest path to the center passes from \vi{p} and \vi{q}, respectively. In (b), when \xk{c}$\in [v_p,b_i]$ where $ b_i $ is the arc bottleneck point, the shortest path passes from \vi{p}; otherwise, the shortest path passes from \vi{q}. The function is piecewise and both the function in $(0,b_i)$ and the function in $(b_i,l)$ are convex by second derivative test. Each piece of $f_i$ is convex and increasing with the distance. Because of the convexity of each $ f_i $, $ f $ is also convex, which implies that $f$ may contain a local minimum along the edge. Therefore, the optimal solution could be found on the edges. 

\begin{figure}[htbp]
	\centering
	\includegraphics[width=1\linewidth]{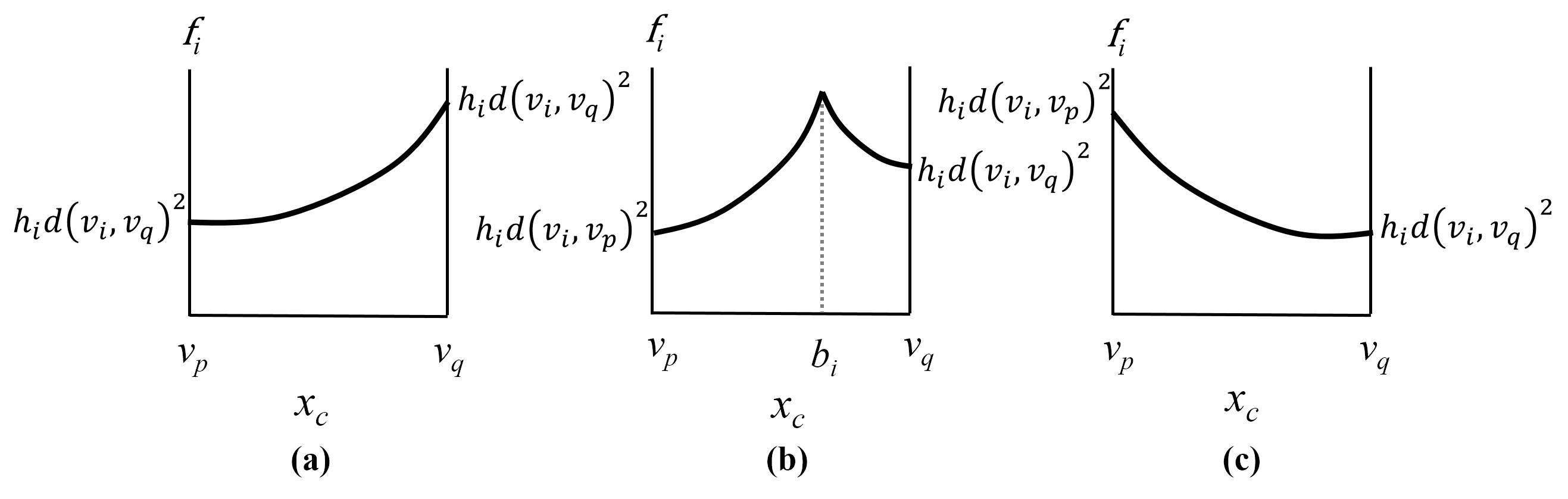}
	\caption{Objective function component for \vi{i} (denoted as $f_i$) in \sscs\, problem with single cluster when \xk{c} is moved along the edge $(v_p,v_q)$}
	\label{fig:ssq1obj}
\end{figure}

\begin{theorem}\label{thm:ssc1}
        In the \sscs\ problem with a single cluster, the optimal cluster center may be located at an interior point of an edge.
	\begin{proof}
		In order to prove this theorem, we will take objective function value of a vertex which has the best objective function value among vertex set V. Then, we will show that under certain conditions, an interior point along the edge will have a lower objective value.
		Let objective function value at vertex $v_p$ be $f^p$. Then,
		\begin{equation*}
		f^p=\sum_{i=1}^N{h_id(v_i,v_p)^2}.
		\end{equation*}
		Assume that the set of points have been arranged such that for points $j=1,...,r$ the shortest path to $v_p$ does not contain the edge $(v_p,v_q)$, and for points $j=r+1,...,N$, the shortest path to $v_p$ contains the edge $(v_p,v_q)$, that is, $d(v_{i_j},v_p)=d(v_{i_j},v_q)+d(v_p,v_q)$. Then, 
		\begin{equation*}
		f_p=\sum_{j=1}^{r}{h_{i_j}d(v_{i_j},v_p)^2}+\sum_{j=r+1}^N{h_{i_j}(d(v_{i_j},v_q)+d(v_p,v_q))^2}.
		\end{equation*}
		There exists a point $x$ on the edge $(v_p,v_q)$ at which the same partitioning is valid. Then, the objective function value at $x$ is
		\begin{equation*}
		f^x=\sum_{j=1}^{r}{h_{i_j}(d(v_{i_j},v_p)+d(v_p,x))^2}+\sum_{j=r+1}^N{h_{i_j}(d(v_{i_j},v_q)+d(v_p,v_q)-d(v_p,x))^2}.
		\end{equation*}
		Rearranging this expression, we have
		\begin{equation*}
		\resizebox{.95\hsize}{!}{$f^x=f^p+d(v_p,x)^2\left[\sum_{j=1}^Nh_{i_j}\right]+2d(v_p,x) \left[\sum_{j=1}^rh_{i_j}d(v_{i_j},v_p) -\sum_{j=r+1}^Nh_{i_j}(d(v_{i_j},v_q)+d(v_p,v_q))\right].$}
		\end{equation*}
		Let
		\begin{align*}
		\mathmodify{a=\sum_{j=1}^Nh_{i_j}},\quad & \mathmodify{b=\sum_{j=1}^rh_{i_j}d(v_{i_j},v_p)-\sum_{j=r+1}^Nh_{i_j}(d(v_{i_j},v_q)+d(v_p,v_q))}.
		\end{align*}
		Then, we have
		\begin{equation}\label{eqzpzx}
		f^x=f^p+ad(v_p,x)^2+2bd(v_p,x).
		\end{equation}
		On the right-hand side, the expression after $f^p$ is a quadratic function of $d(v_p,x)$, that is, \mathmodify{$f(x)=ax^2+2bx$. From \eqref{eqzpzx}, we can say that if $f(x)\leq0$, $f^x \leq f^p$.} If $f(x)>0$, $f^x>f^p$.
		This is possible when $d(v_p,x)\in[0,-2b/a]$. In order for \textit{x} to have a nonempty interval, $-2b/a\geq0$. Since $a$ is nonnegative, this is possible if $b\leq 0$. Hence, the following condition must hold.
		\begin{equation*}
		\sum_{j=1}^rh_{i_j}d(v_{i_j},v_p)\leq \sum_{j=r+1}^Nh_{i_j}(d(v_{i_j},v_q)+d(v_p,v_q))
		\end{equation*}
	\end{proof}
\end{theorem}

\paragraph{\bfseries \sscs\ on Networks with P Clusters}
When $P$ clusters exist, each vertex is assigned to its closest cluster center, minimizing the sum of squared distances. As cluster centers move along edges, assignments may change at assignment bottleneck points, as illustrated in Figure \ref{fig:ssqpobj}.

\begin{figure}[htbp]
	\centering
	\includegraphics[width=1\linewidth]{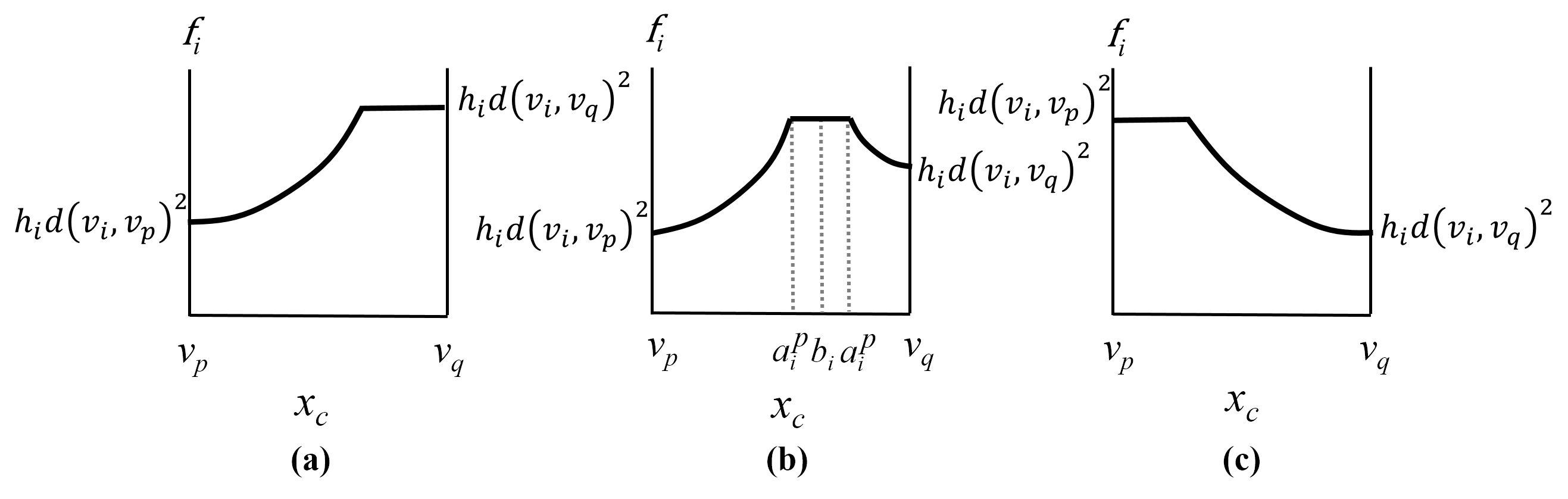}
	\caption{Objective function component for \vi{i} (denoted as $f_i$) in \sscs\, problem with P clusters when \xk{c} is moved along the edge $(v_p,v_q)$ and \xk{k} is the second closest cluster center to \vi{i}}
	\label{fig:ssqpobj}
\end{figure}

\begin{theorem}\label{thm:sscp}
	Let $\mathbf{V^*}$ be a set of P vertices $\left \{ v_1^*,v_2^*,...,v_P^* \right \}$ which is the optimal solution among all possible \Vs sets. There exists a subset $\mathbf{X^*} \in$ \Gr\, containing s ($ s \leq P $) centers located on edges and the remaining centers at vertices $\left \{ x_c,x_2,...,x_s\right \}$ $\left \{v_{s+1}^*,v_{s+2}^*,...,v_P^* \right \}$, such that
	\begin{equation*}
	\sum_{i=1}^{N}{h_{i}d(v_{i},\mathbf{V^*})}^2 \geq \sum_{i=1}^{N}{h_{i}d(v_{i},\mathmodify{\mathbf{X^*}})}^2
	\end{equation*}
\end{theorem}
\ipekmodify{\begin{proof}[Proof sketch]
Rearranging the vertices by their closest center decomposes the objective into per-cluster sums $f_1,...,f_P$, each a single-cluster \sscs\ objective over its own vertex set. By Theorem~\ref{thm:ssc1}, whenever the interior-optimum condition holds for a cluster, its center can be relocated from the vertex to an interior point of an adjacent edge without increasing its sum. Applying this to the $s$ clusters satisfying the condition and observing that the subsequent reassignment of vertices to the new centers can only decrease the objective further yields the claim. The complete derivation is given in~\ref{app:sscp}.
\end{proof}}
This theorem reinforces that restricting cluster centers to vertices may lead to suboptimal solutions in the \sscs\ problem.


\subsection{Soft Assignment Problems}
In this section, two clustering problems that perform soft assignment will be discussed. In soft assignment, each vertex is assigned to all clusters with a probability. There are two soft clustering problems defined in the scope of this study. Both of these problems have been studied on the plane in the literature, and to the best of our knowledge, they have not been studied on networks before. These two problems are called as Probabilistic Distance Clustering (\pdc) and Fuzzy Clustering (\fcm). These problems differ in the objective functions and membership functions they use. As a result of this difference, they have different properties, which will be discussed in further detail.

\subsubsection{Probabilistic Distance Clustering (\pdc) Problem on Networks}

On a network, \pdcp\ is defined as
\begin{align}
\text{\textit{minimize}} \quad & f(\mathbf{X})=\sum_{i=1}^N\sum_{k=1}^P {p_{ik}^2d(v_i,x_k)} \label{objF1}\\
&\text{subject to } \nonumber \\
& \sum_{k=1}^P{p_{ik}=1\quad  \forall \, i=1,...,N},\nonumber \\
& p_{ik} \geq 0 \quad  \forall \, i=1,...,N, \, k=1,...,P,\nonumber\\
& x_k \in \mathbf{G} \quad \forall \, k=1,...,P, \nonumber\\
\end{align}
where \xk{k} is the location of cluster center \( k \) and \( p_{ik} \) is the membership value of \vi{i} to cluster \( k \). For each vertex, the summation of memberships to all clusters must be equal to 1. As shown in \cite{IyigunThesis} by Iyigun and Ben-Israel, using the Lagrangian Method and keeping all \xk{k} fixed, the optimal membership function is

\begin{equation}\label{optp}
p_{ik}^*=\frac{1}{\sum_{l=1}^P{\frac{d(v_i,x_k)}{d(v_i,x_l)}}}.
\end{equation}

Analyzing this problem, we observe that the optimal center locations are on \Vs, which will be proven in this subsection.

\paragraph{\bfseries \pdc\ on Networks with a Single Cluster}\label{ss:PD1}

In \pdcp, when there is one cluster, each vertex will have a membership value of 1 to that cluster. As a result, the problem becomes similar to 1-median problem. $f_i$, the objective function component of \vi{i}, is as illustrated in Figure \ref{fig:pmed1obj} in the example illustrated in Figure \ref{fig:grgeneral}. It is linear and piecewise concave, and its behavior changes at arc bottleneck point if \xk{c} is moved along an edge $(v_p,v_q)$. If there is an arc bottleneck point on an edge as in (b), $f_i$ is linear and piecewise concave along the edge such that it has its maximum at the arc bottleneck point. If there are no arc bottleneck points as in (a) and (c), the distance function is linear. Summation of piecewise concave and linear functions is linear and piecewise concave as well, which is the objective function \eqref{objF1}. Therefore, locating the center to an interior point of an edge will lead higher objective function values. The theorem and its proof is given below.

\begin{theorem}
	In single center \pdcp, \Vs\ contains the set of optimal solutions.
	\begin{proof}
		To prove this theorem, it will be shown that an interior point $x$ on an edge $(v_p,v_q)$ could not have a lower objective value than the vertex \vi{p} which has the best objective value among vertex set \Vs.
		Let objective function value at vertex $v_p$ be $f^p$. Then,
		\begin{equation*}
		f^p=\sum_{i=1}^{N}{p_i^2d(v_i,v_p)}.
		\end{equation*}
		Assume that the set of vertices arranged as the ones whose shortest path to \vi{p} contains the edge $ (v_p,v_q) $ or not. Let $v_{i_j}, j=1,...,r$ show the vertices that does not contain the edge $(v_p,v_q)$, and $j=r+1,...,N$ show the ones that contains the edge $(v_p,v_q)$, that is, $d(v_{i_j},v_p)=d(v_{i_j},v_q)+d(v_p,v_q)$. Then, 
		\begin{equation*}
		f^p=\sum_{j=1}^{r}{p_{i_j}^2d(v_{i_j},v_p)}+\sum_{j=r+1}^{N}{p_{i_j}^2(d(v_{i_j},v_q)+d(v_p,v_q))}.
		\end{equation*}
		There exists a center $x$ on the edge $(v_p,v_q)$ at which the same arrangement is valid. Then, the objective function value at $x$ is
		\begin{equation*}
		f^x=\sum_{j=1}^{r}{p_{i_j}^2(d(v_{i_j},v_p)+d(v_p,x))}+\sum_{j=r+1}^{N}{p_{i_j}^2((d(v_{i_j},v_q)+d(v_p,v_q)-d(v_p,x)))}.
		\end{equation*}
		Rearranging this expression, we have
		\begin{equation*}
		f^x=f^p+d(v_p,x)\left[\sum_{j=1}^rp_{i_j}^2-\sum_{j=r+1}^Np_{i_j}^2\right].
		\end{equation*}
		\begin{equation} \label{case1Eq}
		\sum_{j=1}^rp_{i_j}^2\geq\sum_{j=r+1}^Np_{i_j}^2 \implies f^p\leq f^x .
		\end{equation}
		Suppose that
		\begin{equation}\label{case2Eq}
		\sum_{j=1}^rp_{i_j}^2<\sum_{j=r+1}^Np_{i_j}^2.
		\end{equation}
		Since $d(v_p,v_q)$ is a positive constant, multiplying both sides with $d(v_p,v_q)$, we have
		\begin{equation*}
		d(v_p,v_q)\sum_{j=1}^rp_{i_j}^2<d(v_p,v_q)\sum_{j=r+1}^Np_{i_j}^2.
		\end{equation*}
		We may rewrite $f^p$ as
		\begin{equation*}
		f^p=\sum_{j=1}^{r}{p_{i_j}^2d(v_{i_j},v_p)}+\sum_{j=r+1}^{N}{p_{i_j}^2d(v_{i_j},v_q)}+\sum_{j=r+1}^{N}{p_{i_j}^2d(v_p,v_q)}.
		\end{equation*}
		By \eqref{case2Eq}, we may write
		\begin{equation*}
		f^p > \sum_{j=1}^{r}{p_{i_j}^2d(v_{i_j},v_p)}+\sum_{j=r+1}^{N}{p_{i_j}^2d(v_{i_j},v_q)}+\sum_{j=1}^{r}{p_{i_j}^2d(v_p,v_q)}.
		\end{equation*}
		Right-hand side of this inequality is an upper bound to the objective function value at $v_q$. Hence,
		\begin{equation*}
		f^p > f^q,
		\end{equation*}
		which contradicts with \vi{p}'s having the minimum objective value among all vertices. This implies that $f^x$ will always be greater than $f^p$. Therefore, the optimal location will always be on a vertex.
	\end{proof}
\end{theorem}

In the following subsection, a generalized version of this problem, \pdcp, with $ P $ clusters will be analyzed.

\paragraph{\bfseries \pdc\ on Networks with P Clusters}\label{ss:pdPclus}

When there are $ P $ clusters, \pdc\ works with membership values which depend on location of centers. In this subsection, it will be proven that in the optimal solution, centers of a \pdcp\ with $ P $ clusters on a network will be on vertices. 

For the sake of simplicity, suppose we have two clusters. If \eqref{optp} is evaluated for this case, membership value of vertex $i$ will be

\begin{align}\label{memberforTwo}
p_{i1}&=\frac{d(v_i,x_2)}{d(v_i,x_c)+d(v_i,x_2)}, &
p_{i2}&=\frac{d(v_i,x_c)}{d(v_i,x_c)+d(v_i,x_2)}. 
\end{align}

\begin{figure}[htbp]
	\centering
	\includegraphics[width=1\linewidth]{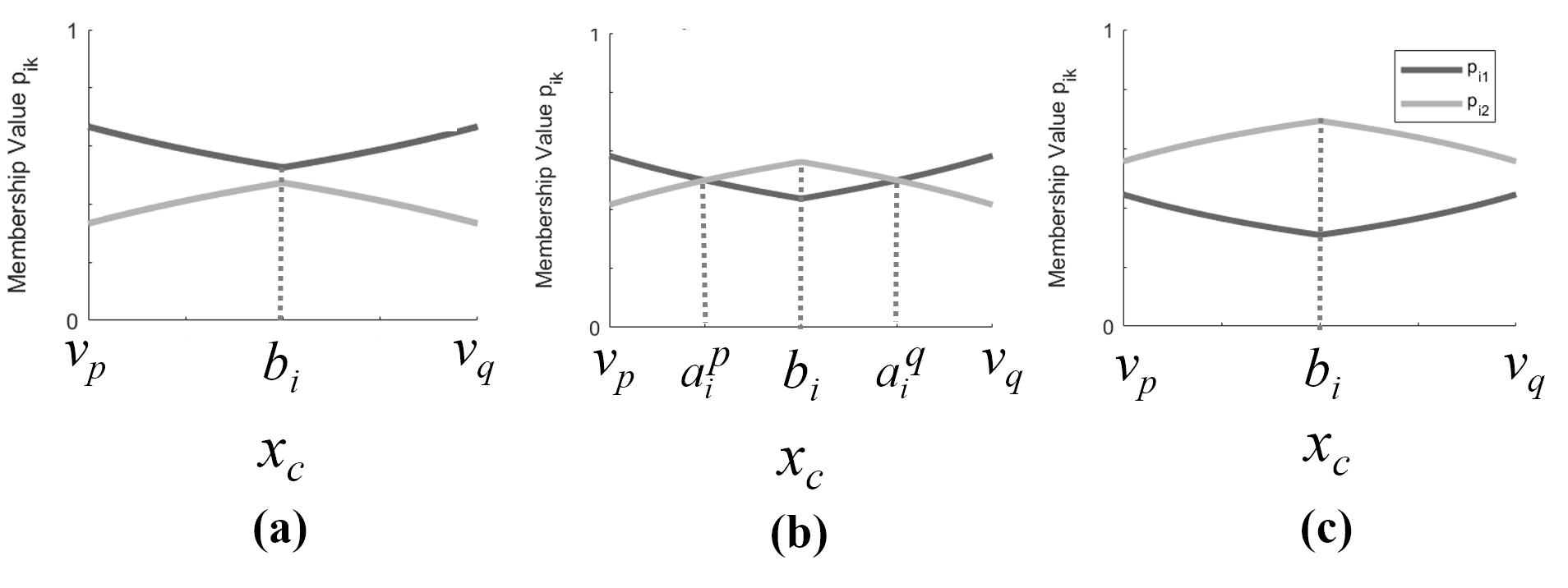}
	\caption{Membership function $p_{ik}$ of \pdc\ with 2 clusters}
	\label{fig:pdmemplot}
\end{figure}

For a graph \Gr\ with two clusters as in Figure \ref{fig:grgeneral}, keeping \xk{2} fixed and moving \xk{c} on the edge $(v_p,v_q)$, change in the membership functions \pik{i}{1} and \pik{i}{2} has been visualized in Figure \ref{fig:pdmemplot}. Although \xk{2} has a fixed location, it is affected from the location change of \xk{c}. In (a), (b) and (c), as \xk{c} moves towards arc bottleneck point $b_i$, \pik{i}{1} decreases since distance increases and reaches the maximum at $b_i$. As \pik{i}{1} decreases, \pik{i}{2} increases. In (a), even at the arc bottleneck point, \dik{i}{1} is less than \dik{i}{2}; therefore, \pik{i}{2} increases but it cannot be greater than \pik{i}{1}. In (c), the contrast occurs. \dik{i}{2} is less than \dik{i}{1} even when \xk{c} is located on endpoints \vi{p} or \vi{q}. Therefore, \pik{i}{1}is always less than \pik{i}{2}. In (b), when  $x_c \in [v_p,a_i^p]$, \pik{i}{1} is greater than \pik{i}{2}. When $x_c \in [a_i^p,a_i^q]$, \pik{i}{2} is greater than \pik{i}{1} since \dik{i}{2} is less than \dik{i}{1}. Lastly, as  $x_c \in [a_i^q,v_q]$, again, \dik{i}{1} decreases and becomes less than \dik{i}{2}. Therefore, \pik{i}{1} is greater than \pik{i}{2}. In (b), assignment bottleneck points could be observed as the points where \pik{i}{1}=\pik{i}{2}.

If \eqref{memberforTwo} is substituted in \eqref{objF1}, the objective function will be
\begin{equation*}\label{obj1For2}
f(\textbf{X})=\sum_{i=1}^N{\frac{d(v_i,x_c)d(v_i,x_2)}{d(v_i,x_c)+d(v_i,x_2)}}.
\end{equation*}
For three clusters, the resulting membership values will be
\begin{align*}
p_{i1}&=\frac{d(v_i,x_2)d(v_i,x_3)}{d(v_i,x_c)d(v_i,x_2)+d(v_i,x_2)d(v_i,x_3)+d(v_i,x_c)d(v_i,x_3)}, \\
p_{i2}&=\frac{d(v_i,x_c)d(v_i,x_3)}{d(v_i,x_c)d(v_i,x_2)+d(v_i,x_2)d(v_i,x_3)+d(v_i,x_c)d(v_i,x_3)}, \\
p_{i3}&=\frac{d(v_i,x_c)d(v_i,x_2)}{d(v_i,x_c)d(v_i,x_2)+d(v_i,x_2)d(v_i,x_3)+d(v_i,x_c)d(v_i,x_3)}.
\end{align*}
With the same manner, the objective function could be written as
\begin{equation*}\label{obj1For3}
f(\textbf{X})=\sum_{i=1}^N{\frac{d(v_i,x_c)d(v_i,x_2)d(v_i,x_3)}{d(v_i,x_c)d(v_i,x_2)+d(v_i,x_2)d(v_i,x_3)+d(v_i,x_c)d(v_i,x_3)}}.
\end{equation*}
For the problem with P clusters, the objective function is
\begin{equation}\label{obj1ForP}
f=\sum_{i=1}^N{\frac{\prod_{k=1}^P\,d(v_i,x_k)}{\sum_{k=1}^P{\prod_{l\neq k}d(v_i,x_l)}}}.
\end{equation}
Since we assume that location of \xk{k} is fixed for $k=2,...,P$, we could separate constant components of each vertex as $K_i$ and write the objective function \eqref{obj1ForP} as in \eqref{obj1Pcons}.
\mathmodify{\begin{align}\label{obj1Pcons}
K_i&=\frac{\prod_{k=2}^P{d(v_i,x_k)}}{\sum_{k=2}^P{\prod_{l\neq k, l \geq 2}d(v_i,x_l)}}\quad\quad \text{and} & f(x_c)&=\sum_{i=1}^N{\frac{d(v_i,x_c)K_i}{d(v_i,x_c)+K_i}}.
\end{align}}
Let the memberships $p_{i1}$ of each point considering the location of cluster center 1 be 
\begin{align} \label{memKi}
p_{i1}&=\frac{K_i}{d(v_i,x_c)+K_i}.
\end{align}

Then, objective function \eqref{obj1Pcons} could be simplified as
\begin{align*}
f(x_c)=\sum_{i=1}^Nd(v_i,x_c)p_{i1}.
\end{align*}

$f_i$, contribution of vertex \vi{i} to the function \eqref{obj1Pcons}, is continuous and twice differentiable. First and second order derivatives are 
\begin{align}
\frac{df_i}{dx_c}&=\frac{K_i^2}{(K_i+d(v_i,x_c))^2} \nonumber \\
\frac{d^2f_i}{dx_c^2}&=\frac{-2K_i^2}{(K_i+d(v_i,x_c))^3} \label{eqn:PDsecder}
\end{align}

With the second derivative test, since \eqref{eqn:PDsecder} is always negative, we can conclude that $f_i$ is concave. 
Let \Gr\,be a graph with P clusters. Keeping \textit{P-1} clusters fixed and moving one cluster center (let us say \xk{c}) along the edge $(v_p,v_q)$, $f_i$ function could be observed as given in Figure \ref{fig:pdpObj}. In (a) and (c), the shortest path from \vi{i} to \xk{c} passes from $v_p$ and $v_q$, respectively. In (b), when $x_c \in [v_p,b_i]$, $f_i$ increases. When $x_c \in [b_i,v_q]$, $f_i$ decreases with the decreasing distance. In (b), $f_i$ is piecewise concave. Different from hard assignment problems, piecewiseness occurs at only arc bottleneck points.  Since each $f_i$ is concave or piecewise concave, $f$, summation of $f_i, \, \forall \, i=1,...,N$ is also concave. 

\begin{figure}[htbp]
	\centering
	\includegraphics[width=1\linewidth]{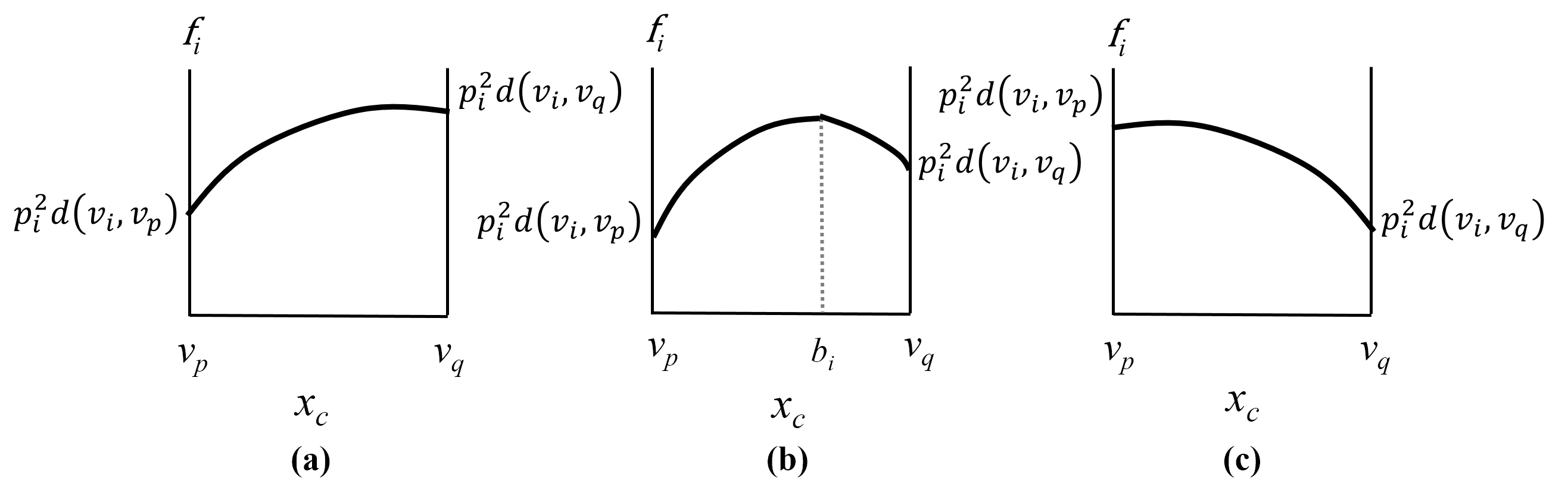}
	\caption{Objective function component for \vi{i} (denoted as $f_i$) in \pdcp\, with P clusters when \xk{c} is moved along the edge $(v_p,v_q)$ and \xk{k} is the second closest cluster center to \vi{i}}
	\label{fig:pdpObj}
\end{figure}

\begin{theorem}\label{thm:pdcp}
	In \pdcp~ with $P$ centers, a cluster center is always at a vertex of a network \Grev.
\end{theorem}
\ipekmodify{\begin{proof}[Proof sketch]
Fix all centers except $x_c$ and let $y_0$ be an arbitrary interior point of an edge $(v_p,v_q)$. Splitting the vertices according to whether their shortest path to $y_0$ enters the edge through $v_p$ or $v_q$, and substituting $d(v_q,y_0)=d(v_p,v_q)-d(v_p,y_0)$, expresses the objective as a function of the single variable $d(v_p,y_0)$. Comparing the two membership-weighted sums induced by this split shows that, depending on which sum dominates, relocating $x_c$ to $v_p$ or to $v_q$ does not increase the objective. Hence no interior point is strictly better than both end vertices, and a cluster center can always be located at a vertex. The complete derivation, covering both cases, is given in~\ref{app:pdcp}.
\end{proof}}

This proof supports Levy's proof that in the case of concavity, the center will be located on \Vs. This result could be generalized such that all the cluster centers are located on vertices in the optimal solution.

\begin{theorem}\label{thm:pdcpall}
	For every cluster center $\{ x_1,x_2,...,x_P\} \in G$, 
	there exists \\ $\{v_{m_1},v_{m_2},...,v_{m_P}\} \in V$ such that 
	\begin{equation*}
	\sum_{k=1}^P{\sum_{i=1}^N{p_{ik}^2d(v_i,x_k)}} \geq \sum_{k=1}^P{\sum_{i=1}^N{p_{ik}^2d(v_i,v_k)}}.
	\end{equation*}
\end{theorem}
\ipekmodify{\begin{proof}[Proof sketch]
Starting from an arbitrary solution, apply Theorem~\ref{thm:pdcp} to each center in turn, keeping the others fixed at their most recent locations. Each relocation to a vertex does not increase the objective, even though the membership values update at every step. Chaining the resulting $P$ inequalities yields the claim. The full chain is given in~\ref{app:pdcp}.
\end{proof}}

As a result, it has been shown that in \pdcp\ on Networks, the optimal solution will always be located on vertices. Therefore, optimum objective function value will not change if the center locations are restricted to \Vs\ instead of \Gr.

\subsubsection{\fc\ (\fcs) Problem on Networks}

\fc\ Problem on Networks is defined as

\begin{align}
\text{\textit{minimize}} \quad & f(\textbf{X})=\sum_{i=1}^N\sum_{k=1}^P {p_{ik}^md(v_i,x_k)^2} \label{eqn:objFuz1}\\
&\text{subject to } \nonumber \\
& \sum_{k=1}^P{p_{ik}=1\quad  \forall \, i=1,...,N},\nonumber \\
& p_{ik} \geq 0 \quad  \forall \, i=1,...,N, \, k=1,...,P,\nonumber\\
& x_k \in G \quad \forall \, k=1,...,P, \nonumber
\end{align}

where \xk{k} is the location of cluster center $ k $ and $p_{ik}$ is the membership value of \vi{i} to cluster k. $ m $ is called \textit{fuzzifier}, or \textit{fuzziness index}. It determines the level of fuzziness in memberships. If $ m=1 $, the problem becomes a hard assignment problem - to be more precise, \sscs\ problem. As $ m $ gets larger, all membership values converge to $1/P$.  For each vertex, summation of memberships to all clusters must be equal to 1. Derived by Bezdek et.al. in \cite{Fuzzy1} with the use of Lagrangian, keeping all \xk{k} fixed, membership function is

\begin{equation}\label{eqn:Foptp}
p_{ik}^*=\frac{1}{\sum_{l=1}^P{\left(\frac{d(v_i,x_k)}{d(v_i,x_l)}\right)^\frac{2}{(m-1)}}}.
\end{equation}

When this problem has been investigated, it has been observed that the optimal center locations could be on anywhere on the \Gr. In this subsection, this property will be analyzed.

\paragraph{\bfseries \fc\ on Networks with a Single Cluster}\label{ss:Fuzzy1}

As in \pdc\, if there is a single cluster, all vertices will have a membership equal to 1. \fcs\ with 1 cluster differs from \pdc\ in that \eqref{eqn:objFuz1}, becomes sum of squared distances. Therefore, the problem will demonstrate characteristics of \sscs\, problem with 1 cluster. In a \Gr\, with one cluster \xk{c} moving along the edge $(v_p,v_q)$, $f_i$, the objective function component of \vi{i}, will be as in Figure \ref{fig:ssq1obj}. $f_i$ is a second degree polynomial function increasing with \dik{i}{k}. $f_i$ is convex or piecewise on an edge. This piecewiseness occur at arc bottleneck points. Each piece of $f_i$ is convex and increasing with distance. Because of the convexity, $f$ which is summation of $f_i$ functions is also convex. But it is not monotone; therefore, it may contain a local minimum along an edge. As a result, there could be an optimal center location located at the interior point of an edge. Based on this observation, we can conclude that the following theorem holds.

\begin{theorem}
Let $\mathbf{V^*}$ be a set of P vertices $\left \{ v_1^*,v_2^*,...,v_p^* \right \}$ which is the optimal solution among all possible \Vs\ sets. In \fcp\ on networks with a single cluster, there may exist a subset $\mathbf{X^*} \in$ \Gr\ containing centers located on edges such that it has an objective function value lower than $\mathbf{V^*}$. 

\begin{proof}
When $P=1$, the membership constraint forces $p_{i1}=1,~\forall~i$, so the objective \eqref{eqn:objFuz1} reduces to $\sum_{i=1}^N{d(v_i,x_1)^2}$, which is the single-cluster \sscs\ objective with unit weights ($h_i=1$). The claim follows from Theorem~\ref{thm:ssc1}.
\end{proof}
\end{theorem} 

\paragraph{\bfseries \fc\ on Networks with P Clusters}\label{ss:Fuzzyp}

In this subsection, objective function of \fcs\ Problem with $ P $ clusters will be analyzed and structural properties will be investigated.

For the sake of simplicity, suppose we have two clusters. If \eqref{eqn:Foptp} is evaluated for this case, membership value of vertex $i$ will be

\begin{align}\label{eqn:FmemberforTwo}
p_{i1}&=\frac{d(v_i,x_2)^\frac{2}{(m-1)}}{d(v_i,x_c)^\frac{2}{(m-1)}+d(v_i,x_2)^\frac{2}{(m-1)}}, &
p_{i2}&=\frac{d(v_i,x_c)^\frac{2}{(m-1)}}{d(v_i,x_c)^\frac{2}{(m-1)}+d(v_i,x_2)^\frac{2}{(m-1)}}. 
\end{align}

For a graph \Gr\ with two clusters as in Figure \ref{fig:grgeneral}, keeping \xk{2} fixed and moving \xk{c} on the edge $(v_p,v_q)$, change in the membership functions \pik{i}{1} and \pik{i}{2} has been visualized in Figure \ref{fig:FCmemplotm2} with fuzziness index $ m $ value of 2. As in \pdc, both memberships are affected by the location change of \xk{c}. In (a), (b) and (c), as \xk{c} moves towards arc bottleneck point $b_i$, \pik{i}{1} decreases due to the increase in distance. As \pik{i}{1} decreases, \pik{i}{2} increases. (a) illustrates the case \dik{i}{1} is less than \dik{i}{2}; therefore, \pik{i}{2} is less than \pik{i}{1}. (c) is the case \dik{i}{2} is less than \dik{i}{1}; as a result, \pik{i}{1} is less than \pik{i}{2}. In (b), if  $x_c \in [v_p,a_i^p]$, \pik{i}{1} is greater than \pik{i}{2}. If $x_c \in [a_i^p,a_i^q]$, \pik{i}{2} is greater than \pik{i}{1}. In the last interval which is  $x_c \in [a_i^q,v_q]$, \dik{i}{1} is less than \dik{i}{2}. Hence, \pik{i}{1} is greater than \pik{i}{2}. In (b), assignment bottleneck points could be observed as the points where \pik{i}{1}=\pik{i}{2}. Figure \ref{fig:FCmemplotm20} is drawn with fuzziness index $m=20$ to illustrate effect of increase in $m$ in membership function. As could be observed, it does not change the behavior of the membership function. However, even at the points where \dik{i}{1} values have the maximum difference, memberships \pik{i}{1} and \pik{i}{2} are very close to each other compared to the case of m=2. The effect of increase in $m$ is increase in the fuzziness of the memberships.

\begin{figure}[htbp]
	\centering
	\includegraphics[width=1\linewidth]{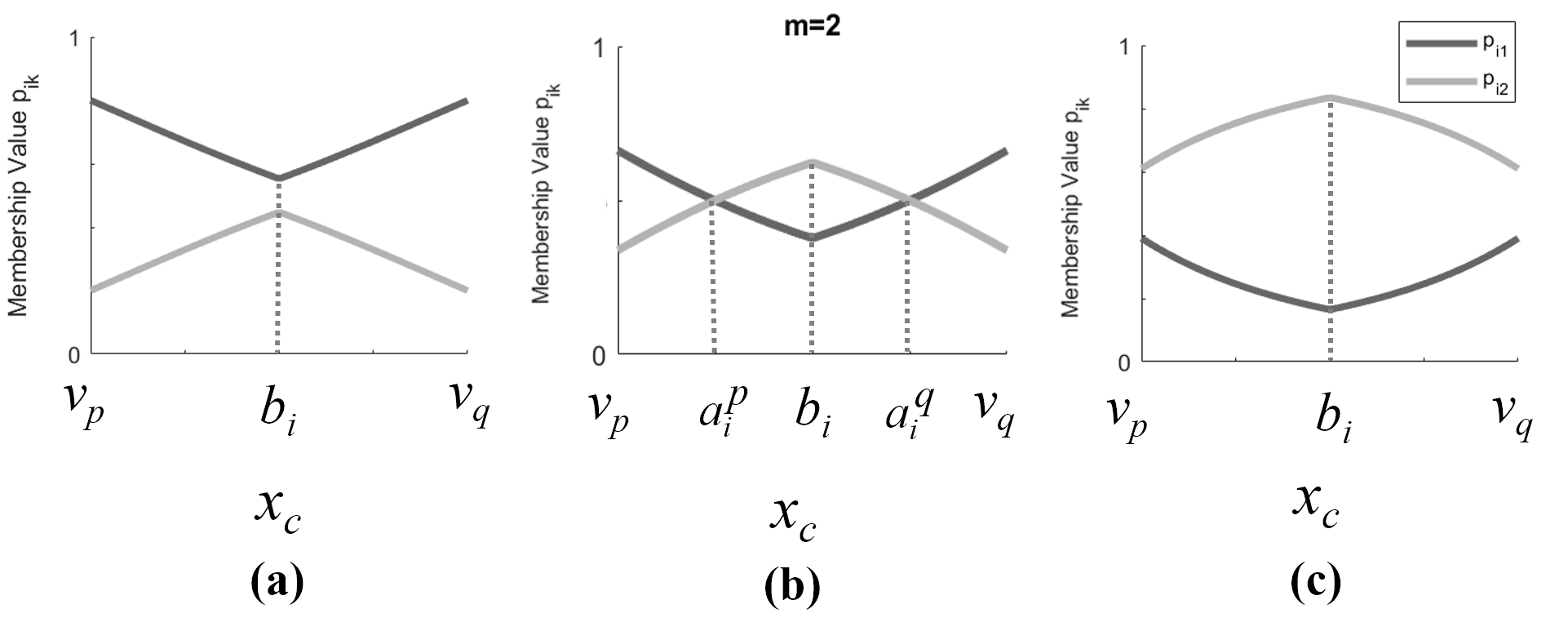}
	\caption{Membership function $p_{ik}$ of \fcs\ with 2 clusters when m=2}
	\label{fig:FCmemplotm2}
\end{figure}

\begin{figure}[htbp]
	\centering
	\includegraphics[width=1\linewidth]{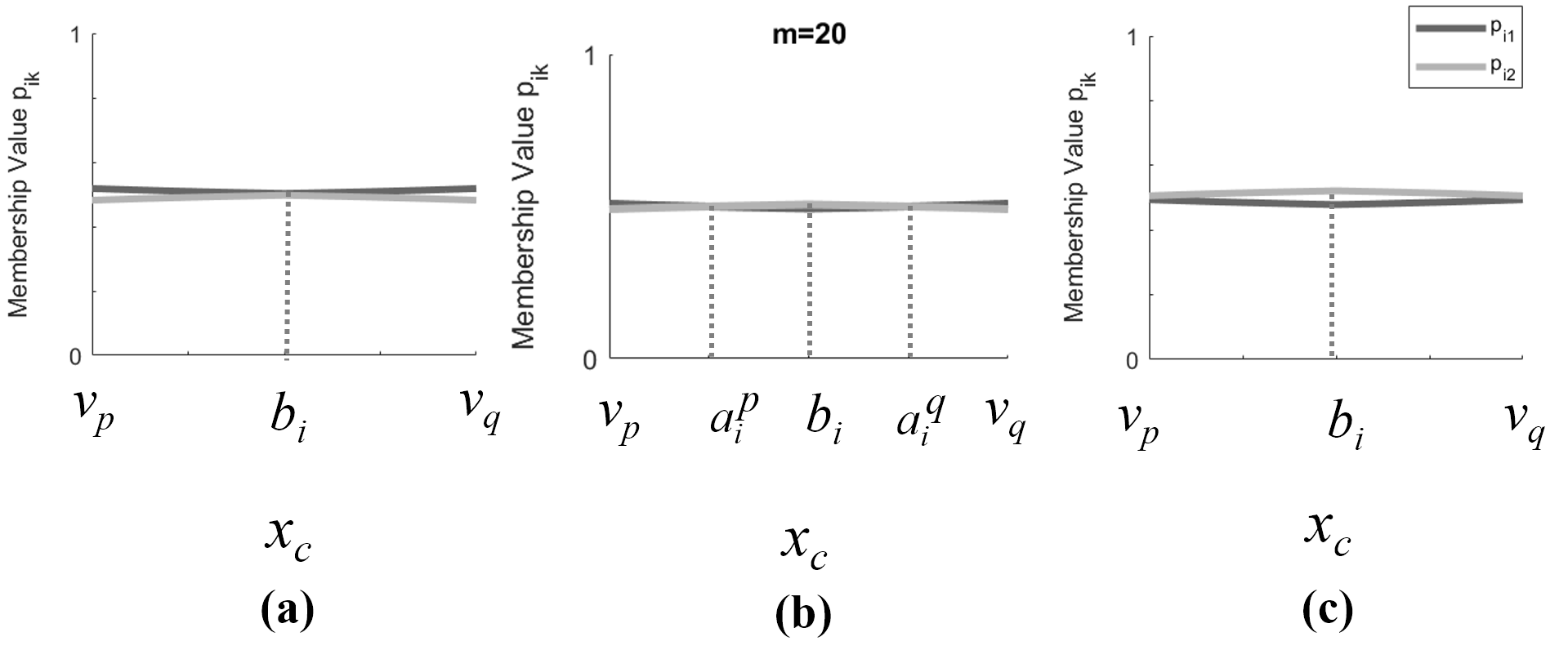}
	\caption{Membership function $p_{ik}$ of \fcs\ with 2 clusters when m=20}
	\label{fig:FCmemplotm20}
\end{figure}

If \eqref{eqn:FmemberforTwo} is substituted in \eqref{eqn:objFuz1}, the objective function will be
\begin{equation*}\label{eqn:Fobj1For2}
f(\textbf{X})=\sum_{i=1}^N{\frac{d(v_i,x_c)^2d(v_i,x_2)^2}{\left(d(v_i,x_c)^\frac{2}{(m-1)}+d(v_i,x_2)^\frac{2}{(m-1)}\right)^{m-1}}}.
\end{equation*}
For three clusters, the membership values are
\begin{align*}
p_{i1}&=\frac{(d(v_i,x_2)d(v_i,x_3))^\frac{2}{(m-1)}}{(d(v_i,x_c)d(v_i,x_2))^\frac{2}{(m-1)}+(d(v_i,x_2)d(v_i,x_3))^\frac{2}{(m-1)}+(d(v_i,x_c)d(v_i,x_3))^\frac{2}{(m-1)}}, \\
p_{i2}&=\frac{(d(v_i,x_c)d(v_i,x_3))^\frac{2}{(m-1)}}{(d(v_i,x_c)d(v_i,x_2))^\frac{2}{(m-1)}+(d(v_i,x_2)d(v_i,x_3))^\frac{2}{(m-1)}+(d(v_i,x_c)d(v_i,x_3))^\frac{2}{(m-1)}}, \\
p_{i3}&=\frac{(d(v_i,x_c)d(v_i,x_2))^\frac{2}{(m-1)}}{(d(v_i,x_c)d(v_i,x_2))^\frac{2}{(m-1)}+(d(v_i,x_2)d(v_i,x_3))^\frac{2}{(m-1)}+(d(v_i,x_c)d(v_i,x_3))^\frac{2}{(m-1)}}.
\end{align*}
With the same manner, the objective function could be written as
\begin{equation*}
f(\textbf{X})=\sum_{i=1}^N{\frac{(d(v_i,x_c)d(v_i,x_2)d(v_i,x_3))^2}{\left((d(v_i,x_c)d(v_i,x_2))^\frac{2}{(m-1)}+(d(v_i,x_2)d(v_i,x_3))^\frac{2}{(m-1)}+(d(v_i,x_c)d(v_i,x_3))^\frac{2}{(m-1)}\right)^{m-1}}}.
\end{equation*}

In the version of the problem with P clusters, the objective function is

\begin{equation}\label{eqn:Fobj1ForP}
f=\sum_{i=1}^N{\frac{\prod_{k=1}^P\,d(v_i,x_k)^2}{\left(\sum_{k=1}^P\,{\prod_{l\neq k}d(v_i,x_l)^\frac{2}{(m-1)}}\right)^{m-1}}}
\end{equation}

For the sake of simplicity, we assume that $m=2$. Since we assume that locations of \xk{k} $\in \mathbf{X} \setminus {x_c}$ are fixed, we could separate constant components of each vertex as $K_i$ and write the objective function \eqref{eqn:Fobj1ForP} as in \eqref{eqn:Fobj1Pcons}.

\begin{align}\label{eqn:Fobj1Pcons}
K_i&=\mathmodify{\frac{\prod_{k=2}^P{d(v_i,x_k)^2}}{\sum_{k=2}^P{\prod_{l\neq k, l \geq 2}d(v_i,x_l)^2}}}\quad\quad \text{and} & f(x_c)&=\sum_{i=1}^N{\frac{d(v_i,x_c)^2K_i}{d(v_i,x_c)^2+K_i}}
\end{align}

$f_i$, contribution of vertex \vi{i} to the function \eqref{eqn:Fobj1Pcons}, is continuous and twice differentiable. First and second order derivatives are 

\begin{align}
\frac{df_i}{dx_c}&=\frac{2K_i^2}{(K_i+d(v_i,x_c)^2)^2} \nonumber \\
\frac{d^2f_i}{dx_c^2}&=\frac{(2K_i^2)(K_i-3d(v_i,x_c)^2)}{(K_i+d(v_i,x_c)^2)^3}. \label{eqn:FCsecder}
\end{align}

With the second derivative test, $f_i$ is

\begin{itemize}
	\item Convex if $\frac{d^2f_i}{dx_c^2} \geq 0$, that is, $d(v_i,x_c) \leq \sqrt{\frac{K_i}{3}}$,
	\item \mathmodify{Concave if $\frac{d^2f_i}{dx_c^2} \leq 0$, that is, $d(v_i,x_c) \ge \sqrt{\frac{K_i}{3}}$.}	
\end{itemize}

\begin{theorem}\label{thm:fcmp}

In \fcp\, with \textit{P} clusters, given a solution \Xs$ =\{x_1,x_2,...,x_P\} $, if $ x_c \in  $ \Xs\,is moved along a given edge keeping the centers \Xs $\setminus \{x_c\}$ fixed, there could be a location on the given edge that minimizes the objective function \eqref{eqn:Fobj1ForP}.

\end{theorem}
\ipekmodify{\begin{proof}[Proof sketch]
With all centers except $x_c$ fixed, substituting the optimal memberships \eqref{eqn:Foptp} into the objective reduces each component $f_i$ to a univariate function of $d(v_i,x_c)$. Differentiating twice shows that each piece of $f_i$, delimited by arc bottleneck points, is convex below a critical distance threshold and concave beyond it, so $f_i$, and therefore the sum $f$, is in general neither convex nor concave along the edge. Since $f$ is not monotone either, a local minimum can occur at an interior point at which the first derivative vanishes and the second derivative is positive; the concavity argument that forces vertex optimality in \pmp\ and \pdcp\ does not apply. The full derivative analysis is given in ~\ref{app:fcmp}.
\end{proof}}

\section{Discussion}
The theoretical findings presented in this study offer a fresh perspective on clustering and network location problems by revealing structural properties that influence optimal solutions. The results highlight the impact of assignment strategies—whether hard or soft—on where cluster centers should be located within a network. These findings challenge conventional assumptions, particularly the idea that optimal clustering solutions can be freely placed on the network space. Instead, we have shown that in certain cases, such as in the \pdc\ problem, the optimal solution is always restricted to network vertices, while in others, such as the \sscs\ problem, optimal locations may exist on the edges as well.

A key insight of this study is the introduction of assignment bottleneck points, which dictate how assignments shift as cluster centers move along network edges. This property fundamentally changes how clustering solutions should be approached, as it determines regions where abrupt assignment changes occur. Additionally, the verification of vertex optimality in soft assignment clustering problems provides theoretical justification for restricting solution searches to discrete network points rather than treating the network as a continuous space. These properties not only refine our understanding of network clustering but also suggest efficient ways to structure solution algorithms.

From an optimization standpoint, these findings open up new avenues for designing more effective heuristics and metaheuristics. The fact that certain clustering problems always yield optimal solutions at vertices implies that computational effort can be significantly reduced by focusing searches on discrete locations rather than the entire network. For problems where optimal solutions may exist on edges, incorporating assignment bottleneck points into search strategies can lead to improved heuristics that avoid exhaustive enumeration while still capturing high-quality solutions. These theoretical insights provide a foundation for developing practical algorithms that balance efficiency and accuracy, especially in large-scale networks.

The implications of these findings extend beyond theoretical optimization, with direct applicability to real-world problems where clustering decisions must be made on constrained networks. \ipekmodify{In these settings, the relevant notion of distance is precisely the shortest path distance on the underlying network, whether measured in road distance, travel time, or transmission cost, so the framework applies without modification.} In urban planning, these results can inform optimal placement strategies for public service facilities, such as fire stations or hospitals, by ensuring they are located in positions that minimize response times. In telecommunications and sensor networks, where signals must be optimally distributed across a predefined infrastructure, understanding whether solutions should be placed at vertices or along edges can improve network efficiency. Similarly, in supply chain logistics, warehouse and distribution center locations can be determined more effectively by incorporating assignment dynamics into network-based clustering models. The ability to leverage these theoretical properties in real-world settings has the potential to improve decision-making processes in various domains where spatial optimization is critical.

\ipekmodify{The results are likewise independent of the dimensionality of any ambient representation space. Once pairwise dissimilarities are encoded as edge lengths, all subsequent analysis operates on the induced network distance, so the structural properties hold whether the underlying data are two-dimensional coordinates or high-dimensional embeddings. This is particularly relevant for graph-based methods in high-dimensional similarity search, such as navigable small-world graphs~\cite{HNSW}, where the network abstraction is precisely the device used to work around the difficulties of high-dimensional geometry.}

\ipekmodify{These observations acquire practical weight in retrieval pipelines that underpin modern generative systems. Retrieval-augmented generation~\cite{RAG} rests on nearest-neighbor search over embedding collections, which state-of-the-art systems organize as navigable graphs~\cite{HNSW}, and cluster-based partitioning of such collections is a standard indexing device. On an embedding graph, only vertices are realizable center locations, and practitioners restrict candidate centers to data points as a matter of course. Our results qualify this practice in both directions: for distance-based objectives the restriction is provably without loss, by the vertex optimality of \pmp\ and \pdcp, while for squared-distance and soft objectives the optimum may lie strictly inside an edge, so the restriction discards solutions that our bottleneck-point machinery locates. Representative-node selection on knowledge graphs in graph-based retrieval~\cite{GraphRAG} is likewise a median problem on a network. We view a rigorous transfer of these results to large-scale retrieval graphs as a promising direction.}

\ipekmodify{\subsection{Algorithmic implications of bottleneck points}
\label{sec:algimp}
The structural results translate into concrete devices for algorithm design. Fix all centers but one and let the free center move along an edge $e_{pq}$. The arc bottleneck points, of which there are at most $N-2$ on the edge, and the assignment bottleneck points, at most $2(N-2)+2$ by Proposition~\ref{prop:abpbound}, together partition $e_{pq}$ into $O(N)$ subintervals. Within a subinterval, every vertex retains both its cluster assignment and the end vertex through which its shortest path to the moving center passes. All nonsmoothness of the objective is therefore concentrated at these breakpoints.

For \ssc, the objective restricted to a subinterval is a univariate second-degree polynomial, so its minimum over the subinterval is available in closed form, at the unique stationary point or at a subinterval endpoint. The resulting $O(N)$ candidate minimizers per edge constitute a finite dominating set for the single-center subproblem in the sense of Hooker et al.~\cite{FDS1}, and an exact search over one edge reduces to evaluating these candidates, or $O(N|\mathbf{E}|)$ candidates over the whole network. For \fcm, the objective is piecewise only at arc bottleneck points, and its pieces are in general neither convex nor concave, so a local minimum may lie in the interior of an edge, as established in Theorem~\ref{thm:fcmp}. This is precisely why the search cannot be restricted to vertices, and it motivates a subinterval-by-subinterval univariate search, for which derivative-free methods such as golden section search are natural candidates.

The bottleneck points also yield pruning rules for Lloyd-type alternating schemes such as K-Means and K-Median adapted to networks. As long as a center update remains within a subinterval delimited by assignment bottleneck points, no vertex changes its nearest center, so the reassignment step can be skipped or updated incrementally. For \pmp\ and \pdcp, vertex optimality permits a stronger reduction: the search space collapses to vertex set \Vs, turning the continuous problem into a discrete one that K-Medoids-style swap procedures address directly. Given the matrix of shortest path distances between vertices, the bottleneck points of an edge can be computed and sorted in $O(N \log N)$ time, so these reductions are inexpensive relative to the cost of the search they accelerate. We present these observations as directions for algorithm design; developing and benchmarking complete algorithms on these foundations is beyond the structural scope of this paper.}

\ipekmodify{\subsection{Limitations of this study}
Limitations of the present analysis should be stated explicitly. First, all results are derived with respect to the shortest path distance induced by strictly positive edge lengths on an undirected, connected graph. The proofs rely on the metric properties of this distance, in particular the triangle inequality. Dissimilarity measures that violate these properties fall outside the scope of our framework, and directed networks, where distances need not be symmetric, remain an open direction. Second, our contribution is structural rather than algorithmic. We characterize where optimal cluster centers can be located, but we do not report complexity results or computational experiments; the directions in Section~\ref{sec:algimp} are intended as starting points for such work. Third, in the case of soft assignment, the results depend on the specific membership forms of \pdc\ and \fcm ; alternative membership models may behave differently. Lastly, the structural properties we establish, including vertex optimality and the bounds on the number of assignment bottleneck points, hold independently of the relative sizes of clusters. How cluster imbalance affects the practical performance of the algorithms built on these properties is a separate question, which our analysis does not address.}

\section{Conclusion}
This study establishes fundamental theoretical properties that govern clustering problems on networks, leading to a more structured understanding of optimal center placement. The identification of assignment bottleneck points and the distinction between vertex-restricted and edge-admissible solutions provide a new perspective on how clustering problems should be formulated and solved. By demonstrating that probabilistic distance clustering (\pdc) problems always have optimal solutions at vertices, we provide theoretical justification for discrete optimization approaches in these settings. Conversely, by showing that sum of squares clustering (\sscs) solutions may lie on edges, we emphasize the importance of considering a broader search space in these cases.

These findings suggest that traditional plane-based clustering methods, which assume free placement of centers, cannot always be directly applied to networks. Instead, clustering problems in network-constrained environments require careful consideration of assignment dynamics and structural properties. The results of this study can guide the development of more efficient solution approaches, including heuristic and metaheuristic methods that take advantage of these theoretical insights.

Future research should explore algorithmic techniques that leverage the properties identified in this study to develop practical, scalable solutions. Given the broad applicability of clustering problems across transportation, telecommunications, emergency response, and logistics, understanding the fundamental structure of optimal solutions can lead to better decision-making in real-world scenarios. Building on these theoretical results, future studies can bridge the gap between mathematical optimization and practical implementation, ensuring that network-based clustering problems are addressed with both precision and computational efficiency.

\appendix

\ipekmodify{\section{Proof of Theorem~\ref{thm:sscp}}\label{app:sscp}}

	\begin{proof}
		Let $\left \{ v_1^*,v_2^*,...,v_p^* \right \}$ be the set of points in $\mathbf{V^*}$. If these points are rearranged such that
		\begin{align*}
		d(v_{i_j},\mathbf{V^*})=&d(v_{i_j},v_1^*) \quad \forall j=1,...n_1, \\
		d(v_{i_j},\mathbf{V^*})=&d(v_{i_j},v_2^*) \quad \forall j=n_1+1,...n_2,\\
		...&\\
		d(v_{i_j},\mathbf{V^*})=&d(v_{i_j},v_p^*) \quad \forall j=n_{p-1}+1,...n_p, (where\ n_p=N) \\
		\end{align*} 
        the objective function could be written as
		\begin{multline*}
		f=\sum_{j=1}^{n_1}{h_{i_j}d(v_{i_j},v_1^*)}^2+
		\sum_{j=n_1+1}^{n_2}{h_{i_j}d(v_{i_j},v_2^*)}^2+...+
		\sum_{j=n_{p-1}+1}^{n_p}{h_{i_j}d(v_{i_j},v_p^*)}^2
		\end{multline*}
		Let
		\begin{align*}
		f_1=&\sum_{j=1}^{n_1}{h_{i_j}d(v_{i_j},v_1^*)}^2 \\
		f_2=&\sum_{j=n_1+1}^{n_2}{h_{i_j}d(v_{i_j},v_2^*)}^2 \\
		...&\\
		f_p=&\sum_{j=n_{p-1}+1}^{n_p}{h_{i_j}d(v_{i_j},v_p^*)}^2.
		\end{align*}
		Define \mathmodify{$h'_{i_j}=h_{i_j}$} for $j=1,...,n_1$ and $h'_{i_j}=0$ for $j=n_1+1,...,n_p(where\  n_p=N)$, we have
		\begin{equation*}
		f_1=\sum_{j=1}^{N}{h_{i_j}d(v_{i_j},v_1^*)}^2.
		\end{equation*}
		In the previous theorem, we have shown that given a condition, there exists an interior point $x_c$ on an edge adjacent to $v_1^*$ such that
		\begin{equation*}
		f_1 \geq \sum_{j=1}^{N}{h'_{i_j}d(v_{i_j},x_c)}^2,
		\end{equation*}
		which could be written as
		\begin{equation*}
		f_1 \geq \sum_{j=1}^{n_1}{h_{i_j}d(v_{i_j},x_c)}^2.
		\end{equation*}		
		Assume that for cluster centers $\left \{ x_c,x_2,...,x_s\right \}$ this condition is satisfied. Then,	
		\begin{align*}
		f_1 \geq& \sum_{j=1}^{n_1}{h_{i_j}d(v_{i_j},x_c)}^2, \\
		f_2 \geq& \sum_{j=n_1+1}^{n_2}{h_{i_j}d(v_{i_j},x_2)}^2, \\
		...&\\
		f_s \geq& \sum_{j=n_{s-1}+1}^{n_s}{h_{i_j}d(v_{i_j},x_s)}^2. \\
		\end{align*}
    Adding both sides of the inequalities, we have
		\begin{align} \label{eqnDerya}
		f \geq \sum_{j=1}^{n_1}{h_{i_j}d(v_{i_j},x_c)}^2+
		&\sum_{j=n_1+1}^{n_2}{h_{i_j}d(v_{i_j},x_2)}^2+...+
		\sum_{j=n_{s-1}+1}^{n_s}{h_{i_j}d(v_{i_j},x_s)}^2 \nonumber\\
		&+\sum_{j=n_{s}+1}^{n_{s+1}}{h_{i_j}d(v_{i_j},v_{s+1}^*)}^2+
		\sum_{j=n_{s+1}+1}^{n_{s+2}}{h_{i_j}d(v_{i_j},v_{s+2}^*)}^2 \nonumber\\
		&+...+\sum_{j=n_{p-1}+1}^{n_{p}}{h_{i_j}d(v_{i_j},v_{p}^*)}^2
		\end{align}
		
		Let the new set of centers $X=\left \{ x_c,x_2,...,x_s,v^*_{s+1},v^*_{s+2},...,v^*_{P} \right \}$. After changing locations of cluster centers, assignments of vertices to centers may change. Therefore, we have
		
		\begin{align} \label{eqnIpek}
		\sum_{j=1}^{n_1}{h_{i_j}d(v_{i_j},x_c)}^2+
		&\sum_{j=n_1+1}^{n_2}{h_{i_j}d(v_{i_j},x_2)}^2+...+
		\sum_{j=n_{s-1}+1}^{n_s}{h_{i_j}d(v_{i_j},x_s)}^2 \nonumber \\
		&+\sum_{j=n_{s}+1}^{n_{s+1}}{h_{i_j}d(v_{i_j},v_{s+1}^*)}^2+
		\sum_{j=n_{s+1}+1}^{n_{s+2}}{h_{i_j}d(v_{i_j},v_{s+2}^*)}^2 \nonumber\\
		&+...+\sum_{j=n_{p-1}+1}^{n_{p}}{h_{i_j}d(v_{i_j},v_{p}^*)}^2 \geq \sum_{j=1}^{N}{h_{i}d(v_{i},\mathmodify{\mathbf{X^*}})}^2.
		\end{align}
		Combining \eqref{eqnDerya} and \eqref{eqnIpek}, we have
		\begin{equation*}
		\sum_{i\in \mathbf{V}}{h_{i}d(v_{i},\mathbf{V^*})}^2 \geq
		\sum_{i\in \mathbf{V}}{h_{i}d(v_{i},\mathmodify{\mathbf{X^*}})}^2.
		\end{equation*}
	\end{proof}

\ipekmodify{\section{Proof of Theorems~\ref{thm:pdcp} and ~\ref{thm:pdcpall}}\label{app:pdcp}}

\begin{proof}Proof of Theorem~\ref{thm:pdcp}

	To prove this theorem, it will be shown that keeping \xk{k} $\in \mathbf{X}\setminus\{x_c\}$ fixed, $x_c$ will always be located on a vertex. 
	Let $y_0$ be an arbitrary point on the edge $(v_p,v_q) \in$ \Es\ and $y_0 \notin V$. There exists a vertex $v_m$ such that
	\begin{equation*}
	\sum_{i=1}^N{\frac{d(v_i,y_0)K_i}{d(v_i,y_0)+K_i}} \geq
	\sum_{i=1}^N{\frac{d(v_i,v_m)K_i}{d(v_i,v_m)+K_i}}.
	\end{equation*}
	We know that
	\begin{equation*}
	d(v_i,y_0)=\min \left \{ d(v_i,v_p)+d(v_p,y_0),d(v_i,v_q)+d(v_q,y_0) \right \}.
	\end{equation*}
	Assume that the set of points have been arranged such that
	\begin{align*}
	d(v_{i_j},y_0)&= d(v_{i_j},v_p)+d(v_p,y_0), for  \ j=1,...,r, \\
	d(v_{i_j},y_0)&= d(v_{i_j},v_q)+d(v_q,y_0), for  \ j=r+1,...,N.
	\end{align*}
	Then, the objective function could be written as
	\begin{multline}\label{ObjProof1}
	\sum_{j=1}^N{\frac{d(v_{i_j},y_0)K_{i_j}}{d(v_{i_j},y_0)+K_{i_j}}}=\sum_{j=1}^r
	{\frac{(d(v_{i_j},v_p)+d(v_p,y_0))K_{i_j}}{d(v_{i_j},v_p)+d(v_p,y_0)+K_{i_j}}} \\
	+\sum_{j=r+1}^N
	{\frac{(d(v_{i_j},v_q)+d(v_q,y_0))K_{i_j}}{d(v_{i_j},v_q)+d(v_q,y_0)+K_{i_j}}}.
	\end{multline}
	Since we have two vertices $v_p$ and $v_q$ connected by the edge which contains $y_0$, either $v_p$ or $v_q$ is a better solution. We will consider two cases each implying that one of the vertices is a better solution.\\
	
	Substitute $d(v_q,y_0)=d(v_p,v_q)-d(v_p,y_0)$ the objective function in \eqref{ObjProof1} is
	\begin{multline*}
	f=\sum_{j=1}^N{\frac{d(v_{i_j},y_0)K_{i_j}}{d(v_{i_j},y_0)+K_{i_j}}}=\sum_{j=1}^r
	{\frac{(d(v_{i_j},v_p)+d(v_p,y_0))K_{i_j}}{d(v_{i_j},v_p)+d(v_p,y_0)+K_{i_j}}} \\
	+\sum_{j=r+1}^N
	{\frac{(d(v_{i_j},v_q)+d(v_p,v_q)-d(v_p,y_0))K_{i_j}}{d(v_{i_j},v_q)+d(v_p,v_q)-d(v_p,y_0)+K_{i_j}}}.
	\end{multline*}
	Let \begin{equation*}
	f=f_1+f_2,
	\end{equation*}where
	\begin{align}
	f_1=&\sum_{j=1}^r{\frac{(d(v_{i_j},v_p)+d(v_p,y_0))K_{i_j}}{d(v_{i_j},v_p)+d(v_p,y_0)+K_{i_j}}}, \label{z1} \\
	f_2=&\sum_{j=r+1}^N
	{\frac{(d(v_{i_j},v_q)+d(v_p,v_q)-d(v_p,y_0))K_{i_j}}{d(v_{i_j},v_q)+d(v_p,v_q)-d(v_p,y_0)+K_{i_j}}}. \label{z2}
	\end{align}
	Multiply both the numerator and denominator of each term of the summation of $f_1$ in \eqref{z1} by $d(v_{i_j},v_p)+K_{i_j}$, then
	\begin{align*}
	f_1=&\sum_{j=1}^r \left[ {\frac{d(v_{i_j},v_p)K_{i_j}}{d(v_{i_j},v_p)+K_{i_j}}+\frac{d(v_p,y_0)K_{i_j}^2}{(d(v_{i_j},v_p)+d(v_p,y_0)+K_{i_j})(d(v_{i_j},v_p)+K_{i_j})}} \right].
	\end{align*}
	Similarly, multiply both the numerator and denominator of each term of the summation of $f_2$ in \eqref{z2} by $(d(v_{i_j},v_q)+d(v_p,v_q)+K_{i_j})$, then
	\begin{multline}\label{z2Sep}
	f_2=\sum_{j=r+1}^N \left[ 
	{\frac{(d(v_{i_j},v_q)+d(v_p,v_q))K_{i_j}}{d(v_{i_j},v_q)+d(v_p,v_q)+K_{i_j}}} \right. \\  - \left. {\frac{d(v_p,y_0)K_{i_j}^2}{(d(v_{i_j},v_q)+d(v_p,v_q)-d(v_p,y_0)+K_{i_j})(d(v_{i_j},v_q)+d(v_p,v_q)+K_{i_j})}} \right].
	\end{multline}
	By triangle inequality, we have $d(v_{i_j},v_q)+d(v_p,v_q) \geq d(v_{i_j},v_p)$. Then
	\begin{equation} \label{eqRel}
	\sum_{j=r+1}^N  
	{\frac{(d(v_{i_j},v_q)+d(v_p,v_q))K_{i_j}}{d(v_{i_j},v_q)+d(v_p,v_q)+K_{i_j}}} \geq
	\sum_{j=r+1}^N \frac{d(v_{i_j},v_p)K_{i_j}}{d(v_{i_j},v_p)+K_{i_j}}.
	\end{equation}
	Substitute right-hand side of \eqref{eqRel} in \eqref{z2Sep}, we have
	\begin{multline} \label{z2prime}
	f_2 \geq f_2'=\sum_{j=r+1}^N \left[ \frac{d(v_{i_j},v_p)K_{i_j}}{d(v_{i_j},v_p)+K_{i_j}} \right. \\ 
	-\left. {\frac{d(v_p,y_0)K_{i_j}^2}{(d(v_{i_j},v_q)+d(v_p,v_q)-d(v_p,y_0)+K_{i_j})(d(v_{i_j},v_q)+d(v_p,v_q)+K_{i_j})}} \right].
	\end{multline}
	Since $d(v_{i_j},v_q)+d(v_p,v_q) \geq d(v_{i_j},v_q)$, replace $d(v_{i_j},v_q)+d(v_p,v_q)$ in \eqref{z2prime} with $d(v_{i_j},v_q)$, then
	\begin{multline*} 
	f_2' \geq f_2''=\sum_{j=r+1}^N \left[ \frac{d(v_{i_j},v_p)K_{i_j}}{d(v_{i_j},v_p)+K_{i_j}} \right. \\ 
	-\left. {\frac{d(v_p,y_0)K_{i_j}^2}{(d(v_{i_j},v_q)+d(v_p,v_q)-d(v_p,y_0)+K_{i_j})(d(v_{i_j},v_q)+K_{i_j})}} \right].
	\end{multline*}
	
	Hence, 
	\begin{equation}\label{zprime}
	f \geq f_1+f_2''
	\end{equation}
	If \eqref{zprime} is rearranged, then
	\begin{align}
	f \geq & \sum_{j=1}^N \frac{d(v_{i_j},v_p)K_{i_j}}{d(v_{i_j},v_p)+K_{i_j}} \label{objPP} \\
	& +d(v_p,y_0) \left[ \sum_{j=1}^r \frac{K_{i_j}^2}{(d(v_{i_j},v_p)+d(v_p,y_0)+K_{i_j})(d(v_{i_j},v_p)+K_{i_j})} \right. \label{objp21} \\
	& - \left. \sum_{j=r+1}^N \frac{K_{i_j}^2}{(d(v_{i_j},v_q)+d(v_p,v_q)-d(v_p,y_0)+K_{i_j})(d(v_{i_j},v_q)+K_{i_j})} \right]. \label{objp22}
	\end{align}
	Summations in \eqref{objp21} and \eqref{objp22} are equal to \eqref{memP1} and \eqref{memP2}, respectively.
	\begin{align}
	\sum_{j=1}^r{p_{i_{j}y_0}p_{i_{j}v_p}}\label{memP1} \\  \sum_{j=r+1}^N{p_{i_{j}y_0}p_{i_{j}v_q}}\label{memP2}
	\end{align}
	If \eqref{Case1} is satisfied, \eqref{pbetter} will hold true. That is, $v_p$ is a better location than $y_0$ for $x_c$. 
	\begin{equation}\label{Case1}
	\sum_{j=1}^r{p_{i_{j}y_0}p_{i_{j}v_p}} \geq \sum_{j=r+1}^N{p_{i_{j}y_0}p_{i_{j}v_q}}
	\end{equation}
	\begin{equation}\label{pbetter}
	\sum_{j=1}^N{\frac{d(v_{i_j},y_0)K_{i_j}}{d(v_{i_j},y_0)+K_{i_j}}} \geq
	\sum_{j=1}^N{\frac{d(v_{i_j},v_p)K_{i_j}}{d(v_{i_j},v_p)+K_{i_j}}}
	\end{equation}
	If \eqref{Case1} is not satisfied, then we have
	\begin{equation}\label{Case2}
	\sum_{j=1}^r{p_{iy_0}p_{iv_p}} < \sum_{j=r+1}^N{p_{iy_0}p_{iv_q}}
	\end{equation}
	Clearly, \eqref{pbetter} is not guaranteed in this case.\\
	Again, let 
	\begin{equation*}
	f=f_1+f_2
	\end{equation*}
	Add and subtract $d(v_p,v_q)$ both numerators and denominators of each term of summation of $f_1$ in \eqref{z1}, then
	\begin{align}
	f_1=&\sum_{j=1}^r{\frac{(d(v_{i_j},v_p)+d(v_p,y_0)+d(v_p,v_q)-d(v_p,v_q))K_{i_j}}{d(v_{i_j},v_p)+d(v_p,y_0)+d(v_p,v_q)-d(v_p,v_q)+K_{i_j}}}, \label{z1-2} 
	\end{align}
	Multiply both the numerator and denominator of each term of the summation of $f_1$ in \eqref{z1-2} by $d(v_{i_j},v_p)+d(v_p,v_q)+K_{i_j}$, then
	\begin{align*}
	f_1=&\sum_{j=1}^r \left[ {\frac{(d(v_{i_j},v_p)+d(v_p,v_q))K_{i_j}}{d(v_{i_j},v_p)+d(v_p,v_q)+K_{i_j}}} \right. \\
	-&\left.\frac{(d(v_p,v_q)-d(v_p,y_0))K_{i_j}^2}{(d(v_{i_j},v_p)+d(v_p,v_q)-d(v_p,v_q)+d(v_p,y_0)+K_{i_j})(d(v_{i_j},v_p)+d(v_p,v_q)+K_{i_j})} \right].
	\end{align*}
	Cancelling the $d(v_p,v_q)-d(v_p,v_q)$ expression, we have
	\begin{align}\label{z1Sep2}
	f_1=&\sum_{j=1}^r \left[ {\frac{(d(v_{i_j},v_p)+d(v_p,v_q))K_{i_j}}{d(v_{i_j},v_p)+d(v_p,v_q)+K_{i_j}}} \right. \\
	-&\left.\frac{(d(v_p,v_q)-d(v_p,y_0))K_{i_j}^2}{(d(v_{i_j},v_p)+d(v_p,y_0)+K_{i_j})(d(v_{i_j},v_p)+d(v_p,v_q)+K_{i_j})} \right].
	\end{align}
	Similarly, multiply both the numerator and denominator of each term of the summation of $f_2$ in \eqref{z2} by $(d(v_{i_j},v_q)+K_{i_j})$, then
	\begin{multline*}
	f_2=\sum_{j=r+1}^N \left[ 
	{\frac{d(v_{i_j},v_q)K_{i_j}}{d(v_{i_j},v_q)+K_{i_j}}} \right. \\+ \left.{\frac{((d(v_p,v_q)-d(v_p,y_0))K_{i_j}^2}{(d(v_{i_j},v_q)+d(v_p,v_q)-d(v_p,y_0)+K_{i_j})(d(v_{i_j},v_q)+K_{i_j})}} \right].
	\end{multline*}
	By triangle inequality, we have $d(v_{i_j},v_p)+d(v_p,v_q) \geq d(v_{i_j},v_q)$. Then
	\begin{equation} \label{eqRel2}
	\sum_{j=1}^r  
	{\frac{(d(v_{i_j},v_p)+d(v_p,v_q))K_{i_j}}{d(v_{i_j},v_p)+d(v_p,v_q)+K_{i_j}}} \geq
	\sum_{j=1}^r \frac{d(v_{i_j},v_q)K_{i_j}}{d(v_{i_j},v_q)+K_{i_j}}.
	\end{equation}
	Substitute right-hand side of \eqref{eqRel2} in \eqref{z1Sep2}, we have
	\begin{multline} \label{z1prime}
	f_1 \geq f_1'=\sum_{j=1}^r \left[ \frac{d(v_{i_j},v_q)K_{i_j}}{d(v_{i_j},v_q)+K_{i_j}} \right. \\ 
	-\left.\frac{(d(v_p,v_q)-d(v_p,y_0))K_{i_j}^2}{(d(v_{i_j},v_p)+d(v_p,y_0)+K_{i_j})(d(v_{i_j},v_p)+d(v_p,v_q)+K_{i_j})} \right].
	\end{multline}
	Since $d(v_{i_j},v_p)+d(v_p,v_q) \geq d(v_{i_j},v_p)$, replace $d(v_{i_j},v_p)+d(v_p,v_q)$ in \eqref{z1prime} with $d(v_{i_j},v_p)$, then
	\begin{equation*} 
	f_1' \geq f_1''=\sum_{j=1}^r \left[ \frac{d(v_{i_j},v_q)K_{i_j}}{d(v_{i_j},v_q)+K_{i_j}} \right. \\ 
	-\left.\frac{(d(v_p,v_q)-d(v_p,y_0))K_{i_j}^2}{(d(v_{i_j},v_p)+d(v_p,y_0)+K_{i_j})(d(v_{i_j},v_p)+K_{i_j})} \right].
	\end{equation*}
	Hence, 
	\begin{equation}\label{zprime2}
	f \geq f_1''+f_2
	\end{equation}
	If \eqref{zprime2} is rearranged, then
	\begin{align}
	f \geq & \sum_{j=1}^N \frac{d(v_{i_j},v_q)K_{i_j}}{d(v_{i_j},v_q)+K_{i_j}}\nonumber \\
	& +(d(v_p,v_q)-d(v_p,y_0)) \\ &*\left[ \sum_{j=r+1}^N \frac{K_{i_j}^2}{(d(v_{i_j},v_q)+d(v_p,v_q)-d(v_p,y_0)+K_{i_j})(d(v_{i_j},v_q)+K_{i_j})} \right. \label{objq21} \\
	& - \left. \sum_{j=1}^r \frac{K_{i_j}^2}{(d(v_{i_j},v_p)+d(v_p,y_0)+K_{i_j})(d(v_{i_j},v_p)+K_{i_j})} \right]. \label{objq22}
	\end{align}
	\mathmodify{Summations in \eqref{objq21} and \eqref{objq22} are equal to \eqref{memP2} and \eqref{memP1}, respectively.}
	If \eqref{Case2} is satisfied, \eqref{qbetter} will hold true. $v_q$ is a better location than $y_0$ for $x_c$.
	\begin{equation}\label{qbetter}
	\sum_{i=1}^N{\frac{d(v_i,y_0)K_i}{d(v_{i_j},y_0)+K_i}} \geq
	\sum_{i=1}^N{\frac{d(v_i,v_q)K_i}{d(v_{i_j},v_q)+K_i}}
	\end{equation}
	As shown above, when \eqref{Case1} is satisfied. $v_p$ is a better location than $y_0$ for $x_c$. Otherwise, (when \eqref{Case2} is satisfied), $v_q$ is a better location than $y_0$ for $x_c$. As a result, the center $x_c$ will always be located on a vertex.
\end{proof}

	\begin{proof}Proof of Theorem~\ref{thm:pdcpall}
    
		To prove this theorem, we will change location of a center $x_i$, by fixing the other centers ($x_k, k \neq i$). In each location change, membership scores will change. Initial membership score will be denoted as $p_{ik}^{(0)}$. Updated membership scores resulting from changing location of cluster center $l$ will be denoted as $p_{ik}^{(l)}$.
		
		Let $x_c$ be the cluster center to be changed, keeping the others fixed. In the previous theorem, we have shown that
		\begin{multline} \label{eqfor1}
		\sum_{k=1}^P{\sum_{i=1}^N{p_{ik}^{(0)^2}d(v_i,x_k)}} \geq  {\sum_{i=1}^N{p_{ik}^{(1)^2}d(v_i,v_{m_1})}} + \sum_{k=2}^P{\sum_{i=1}^N{p_{ik}^{(1)^2}d(v_i,x_k)}}.
		\end{multline} 
		Now, let $x_2$ be the cluster center to be changed, keeping the others fixed in their last locations. Again, we have
		\begin{multline} \label{eqfor2}
		{\sum_{i=1}^N{p_{ik}^{(1)^2}d(v_i,v_{m_1})}} + \sum_{k=2}^P{\sum_{i=1}^N{p_{ik}^{(1)^2}d(v_i,x_k)}}\geq \\ \sum_{k=1}^2{\sum_{i=1}^N{p_{ik}^{(2)^2}d(v_i,v_{m_k})}} + \sum_{k=3}^P{\sum_{i=1}^N{p_{ik}^{(2)^2}d(v_i,x_k)}}.
		\end{multline}
		Perform this process with each \xk{k} as the center location to be changed, in the last expression, we have
		\begin{multline} \label{eqforP}
		\sum_{k=1}^{P-1}{\sum_{i=1}^N{p_{ik}^{(P-1)^2}d(v_i,v_{m_k})}} + {\sum_{i=1}^N{p_{ik}^{(P-1)^2}d(v_i,x_p)}}\geq \sum_{k=1}^P{\sum_{i=1}^N{p_{ik}^{(P)^2}d(v_i,v_{m_k})}}.
		\end{multline}
		Combine \eqref{eqfor1}, \eqref{eqfor2} and \eqref{eqforP}, we have
		\begin{equation*}
		\sum_{k=1}^P{\sum_{i=1}^N{p_{ik}^{(0)^2}d(v_i,x_k)}} \geq
		\sum_{k=1}^P{\sum_{i=1}^N{p_{ik}^{(P)^2}d(v_i,v_{m_k})}},
		\end{equation*}
		which implies that as center locations, $v_{m_1},...,v_{m_P}$ lead to a better solution than $x_1,...,x_P$.
	\end{proof}

\ipekmodify{\section{Proof of Theorem~\ref{thm:fcmp}}\label{app:fcmp}}

\begin{proof}
Let 
\begin{align*}
z_i&=\prod_{k=2}^P{d(v_i,x_k)}\\
t_i&=\mathmodify{\sum_{k=2}^{P} \prod_{\substack{j=2 \\ j\neq k}}^P\,{d(v_i,x_j)^{\frac{2}{m-1}}}}.
\end{align*}

If we fix locations of \xk{k} $\in \mathbf{X}\setminus{}\{x_c\}$ and separate fixed components of $f_i$ from variable components by using $z_i$ and $t_i$, the objective function is \eqref{eqn:fncFuzzyCT}.
\begin{equation}\label{eqn:fncFuzzyCT}
f_i=\mathmodify{\frac{z_i^2d(v_i,x_c)^2}{\left(d(v_i,x_c)^{\frac{2}{m-1}}t_i+z_i^{\frac{2}{m-1}}\right)^{m-1}}}
\end{equation}

If we calculate first and second order derivatives for $f_i$ with any $m>1$, we have

\begin{align}
\frac{df_i}{dx_c}&=\frac{(2z_i^2d(v_i,x_c))}{(z_i+t_id(v_i,x_c)^\frac{2}{m-1})^m} \label{eqn:FCfirder} \\
\frac{d^2f_i}{dx_c^2}&=\frac{(2z_i^2)\left(z_i(m-1)-t_i(m+1)d(v_i,x_c)^\frac{2}{m-1}\right)}{\left(\left(z_i+t_id(v_i,x_c)^\frac{2}{m-1}\right)(m-1)\right)^{m+1}} \label{eqn:FCsecderM}
\end{align}

By the second derivative test, $f_i$ is

\begin{itemize}
	\item \mathmodify{Convex if $\frac{d^2f_i}{dx_c^2} \geq 0$, that is, $d(v_i,x_c) \leq {\left(\frac{z_i(m-1)}{t_i(m+1)}\right)}^{\frac{m-1}{2}}$,}
	\item \mathmodify{Concave if $\frac{d^2f_i}{dx_c^2} < 0$, that is, $d(v_i,x_c) > {\left(\frac{z_i(m-1)}{t_i(m+1)}\right)}^{\frac{m-1}{2}}$.}	
\end{itemize}

As a result, $ f_i $ is a nonconvex function of $d(v_i,x_k)$, and increasing with distance. The objective function could be illustrated as in Figure \ref{fig:FCpObj}. As in \pdc, $f_i$ is piecewise at arc bottleneck points. And each piece of $f_i$ is nonconvex according to the sign of second derivative \eqref{eqn:FCsecderM}. Since summation of nonconvex functions are nonconvex, $f$, summation of $f_i \, \forall \, i=1,...,N$, is nonconvex. Since $f$ is not monotone, local minimum could be found at a point where second derivative is positive and first derivative is zero.

Let \Gr\ be a graph with \textit{P} clusters. Keeping \textit{P-1} cluster centers fixed and moving one cluster center (let us say \xk{c}) along the edge $(v_p,v_q)$, the function $f_i$ could be observed as given in Figure \ref{fig:FCpObj}. In (a) and (c), the shortest path from \vi{i} to \xk{c} passes from $v_p$ and $v_q$, respectively. In (b), when $x_c \in [v_p,b_i]$, $f_i$ increases with the distance $d(v_i,x_c)$; when $x_c \in [b_i,v_q]$, $f_i$ decreases with the decreasing distance. As in \pdcp, piecewiseness occurs only at arc bottleneck points. \mathmodify{However, unlike \pdcp, each piece of $f_i$ is not concave: by the second derivative test in \eqref{eqn:FCsecderM}, $f_i$ is convex for $d(v_i,x_c) \leq \left(\frac{z_i(m-1)}{t_i(m+1)}\right)^{\frac{m-1}{2}}$ and concave beyond this threshold. Hence each piece of $f_i$, and therefore the summation $f=\sum_{i=1}^N f_i$, is in general neither convex nor concave. Since $f$ is nonconvex and not monotone along the edge, a local minimum may occur at an interior point $x_c^* \in (v_p,v_q)$ at which the first derivative of $f$ vanishes and the second derivative is positive. Consequently, the concavity argument that forces vertex optimality in \pmp\ and \pdcp\ does not apply, and the optimal solution of \fcp\ may contain cluster centers located at interior points of edges.}

\begin{figure}[htbp]
	\centering
	\includegraphics[width=1\linewidth]{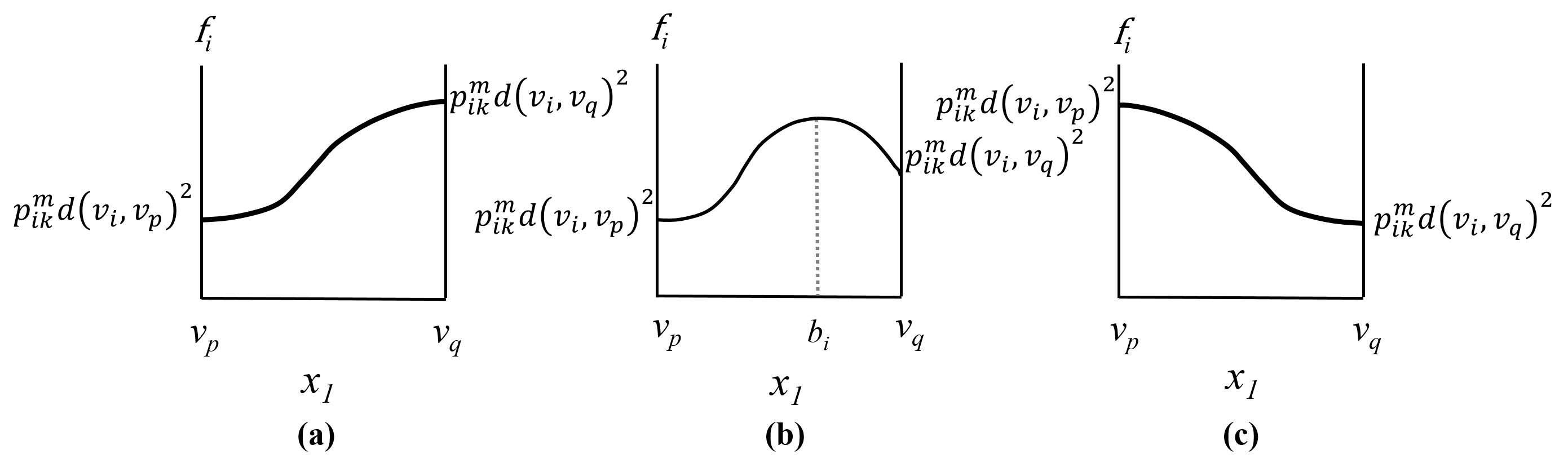}
	\caption{Objective function component for \vi{i} (denoted as $f_i$) in \fcs\ problem with P clusters when \xk{c} is moved along the edge $(v_p,v_q)$ and \xk{k} is the second closest cluster center to \vi{i}}
	\label{fig:FCpObj}
\end{figure}

Hence, in \fcp\, one may find an optimal solution that contains cluster centers located at edges.

\end{proof}

\bibliographystyle{elsarticle-num} 
\bibliography{cas-refs}

\section*{Funding}
No funding was received for conducting this study. This article does not contain any studies with human participants performed by any of the authors.

\section*{Corresponding Author}
Correspondence to Derya Ipek Eroglu

\section*{Conflict of Interest}
The authors declare that they have no conflict of interest and all authors approve this manuscript for publication. 





\end{document}